\documentclass[11pt,twoside,a4paper]{article}
\usepackage[T1]{fontenc}
\usepackage{mathptmx}
\usepackage[colorlinks,citecolor=blue]{hyperref}
\usepackage{cite}
\usepackage{amsthm,amsmath,amssymb}
\usepackage{doi}

\newtheorem{theorem}{Theorem}[section]
\newtheorem{lemma}[theorem]{Lemma}
\newtheorem{corollary}[theorem]{Corollary}
\theoremstyle{definition}
\newtheorem{definition}[theorem]{Definition}

\newcommand{\FF}{\mathsf{F}}
\newcommand{\UU}{\mathsf{U}}
\newcommand{\TT}{\mathsf{T}}
\newcommand{\FUTset}{\{\FF, \UU, \TT\}}
\newcommand{\isdef}[1]{\lceil #1 \rfloor}
\newcommand{\Isdef}[1]{\Big\lceil #1 \Big\rfloor}
\newcommand{\defeq}{:=}
\newcommand{\proves}{\vdash}
\newcommand{\prv}[1]{\mathrel{\vdash_{\textnormal{\tiny #1}}}}
\newcommand{\mc}{\mathcal}
\newcommand{\s}{\;}

\newcommand{\var}{\mathtt{v}}
\newcommand{\arit}{\alpha}
\newcommand{\tulk}[1]{\underline{#1}}
\newcommand{\tulkb}[1]{\underline{#1}_\bot}
\newcommand{\Tulk}[1]{\overline{#1}}
\newcommand{\samef}{\cong}
\newcommand{\yield}{\asymp}
\newcommand{\param}[2]{(#2_1, \ldots, #2_{\arit(#1)})}
\newcommand{\acall}[2]{#1\param{#1}{#2}}
\newcommand{\tcall}[2]{\tulk{#1}\param{#1}{#2}}
\newcommand{\tcallb}[2]{\tulkb{#1}\param{#1}{#2}}
\newcommand{\icall}[2]{\acall{\isdef{#1}}{#2}}
\newcommand{\dom}{\mathbb{D}}
\newcommand{\domb}{\mathbb{D}_\bot}
\newcommand{\itpr}{\tulk{~~}}
\newcommand{\idf}{\isdef{\mc{F}}}
\newcommand{\lukimpl}{\twoheadrightarrow}
\newcommand{\imdl}[2]{$(\sigma, \nu{#2}) \models #1$}

\title{\emph{A Completeness Proof for A Regular Predicate Logic with Undefined
Truth Value}}

\author{Antti Valmari\\
\small{Faculty of Information Technology, University of Jyväskylä, FINLAND}
\and
Lauri Hella\\
\small{Faculty of Information Technology and Communication Sciences,}\\[-.7ex]
\small{Tampere University, FINLAND}}

\date{\scriptsize%
This is the Author’s accepted manuscript (with a different \LaTeX\ style file)
of a paper in\\
Notre Dame J. Formal Logic 64(1): 61--93 (February 2023).
Copyright © 2023 University of Notre Dame\\
\doi{10.1215/00294527-2022-0034},
\href{https://projecteuclid.org/journals/notre-dame-journal-of-formal-logic/volume-64/issue-1/A-Completeness-Proof-for-a-Regular-Predicate-Logic-with-Undefined/10.1215/00294527-2022-0034.short}
{link to the paper at Project Euclid}}

\begin{document}

\maketitle

\pagestyle{myheadings}
\markboth{A.~Valmari and L.~Hella}
{A Complete Regular Predicate Logic with Undefined Truth Value}

\begin{abstract}\noindent
We provide a sound and complete proof system for an extension of Kleene's
ternary logic to predicates.
The concept of theory is extended with, for each function symbol, a formula
that specifies when the function is defined.
The notion of ``is defined'' is extended to terms and formulas via a
straightforward recursive algorithm.
The ``is defined'' formulas are constructed so that they themselves are always
defined.
The completeness proof relies on the Henkin construction.
For each formula, precisely one of the formula, its negation, and the negation
of its ``is defined'' formula is true on the constructed model.
Many other ternary logics in the literature can be reduced to ours.
Partial functions are ubiquitous in computer science and even in (in)equation
solving at schools.
Our work was motivated by an attempt to explain, precisely in terms of logic,
typical informal methods of reasoning in such applications.
\end{abstract}

\bigskip\noindent{\small
Proof theory, general (including proof-theoretic semantics);
Secondary 03B10 Classical first-order logic

}

\section{Introduction}\label{S:Intro}

Classical binary first-order logic assumes that all function symbols denote
total functions.
This assumption repeatedly fails in everyday mathematics and in theoretical
and practical computer science.
As a consequence, there is extensive literature on how to deal with partial
functions in logical reasoning, presenting surprisingly many diverse
approaches, including \cite{ACN88,BCJ84,BB+05,DMR08,deN17,FaG00,GaL90,
GrS95,JoM94,Leh94,Leh01,McC63,Par93,PaG21,ScB99,Spi92,VaH17}.
We faced the problem when developing computer support for school and
elementary university mathematics education~\cite{Val21,VaH17,VaR19}.
To introduce related work and our motivation, it is useful to first present a
somewhat artificial example.

Assume that a student has been asked to find the roots of $3 \sqrt{|x|-1} \geq
x+1$ (that is, solve  $3 \sqrt{|x|-1} \geq x+1$) in the case of real numbers.
One possible way to start is to get rid of the absolute value operator by
splitting the problem to two cases.
Intuitively it seems that this can be expressed with logical connectives as
follows:
\begin{equation}\label{E:split}
(x < 0 \wedge 3 \sqrt{-x-1} \geq x+1) \vee (x \geq 0 \wedge 3 \sqrt{x-1} \geq
x+1)
\end{equation}
The values $x \leq -1$ make $3 \sqrt{-x-1} \geq x+1$ yield $\TT$ (that is,
true), and $-1 < x < 0$ makes it undefined.
Furthermore, $0 \leq x < 1$ makes $3 \sqrt{x-1} \geq x+1$ undefined, $1 \leq x
< 2$ and $x > 5$ result in $\FF$ (false), and $2 \leq x \leq 5$ results in
$\TT$.
Therefore, the student should answer something to the effect of $x \leq -1
\vee 2 \leq x \leq 5$.

To formalize this reasoning in some logic, $3 \sqrt{|x|-1} \geq x+1$ should be
in some sense equivalent to $x \leq -1 \vee 2 \leq x \leq 5$ in that logic.
The trouble begins with the fact that when $-1 < x < 1$, then $3 \sqrt{|x|-1}
\geq x+1$ is undefined but $x \leq -1 \vee 2 \leq x \leq 5$ yields $\FF$.

We could try to sort this out by declaring that the domain of discourse is not
$\mathbb{R}$ but $\{x \in \mathbb{R} \mid |x|-1 \geq 0\}$.
Unfortunately, this idea would invalidate~(\ref{E:split}), because every real
number makes some atomic formula in it undefined, and we obviously want the
roots be in the domain of discourse.
(A similar remark was made in~\cite{Par93}.)
Furthermore, with this kind of exercises it is usually the student's
responsibility to find out that some values are not roots because they make
something undefined.
Those values must be in the domain of discourse of the reasoning whose purpose
is to find them.
In brief, we want a logic that justifies, not bans,~(\ref{E:split}).

Perhaps the next idea is \emph{negative free logic}: every atomic formula that
contains an undefined term yields $\FF$~\cite{FaG00,Nol20,Par93,PaG21}.
It suffers from a problem illustrated by the following example.
If $\geq$ is a relation symbol in the core language and $x < y$ is defined as
an abbreviation of $\neg(x \geq y)$, then $3 \sqrt{|x|-1} < x+1$ yields $\TT$
when $-1 < x < 1$.
It is against our intention.
It is also clumsy: the user of the logic would have to remember which
predicate symbols yield $\FF$ and which yield $\TT$ when applied to undefined
terms.
Defining both $<$ and $\geq$ in the core language would mean to axiomatize
essentially the same thing twice.
It would also result in the loss of the permission to replace $\neg(t \geq
t')$ and $t < t'$ by each other when $t$ and $t'$ are potentially undefined
terms.

In computer science, the idea of \emph{underspecification}~\cite{GrS95} is
popular.
In it every function always yields a value in the intended range, but we
refuse to say anything else about the value when it would be undefined in the
everyday mathematics sense.
From~\cite{GrS95}:
``The value of $x/0$ could be $2.4$ or $20 \cdot x$ or any value in
$\mathbb{R}$; we simply don't say what it is.''
It has a variant where functions may be proper partial, but every relation
symbol always yields $\FF$ or $\TT$.
In the words of~\cite[Sect.\ 2.5]{Spi92}:
``If one or both of $E_1$ and $E_2$ are undefined, then we say that the
predicates $E_1 = E_2$ and $E_1 \in E_2$ are \emph{undetermined}: we do not
know whether they are true or false.
This does not mean that the predicates have some intermediate status in which
they are `neither true nor false', simply that we have chosen not to say
whether they are true or not.''

While underspecification works well for proving programs correct, it is
unsuitable for our purpose.
By its very nature, it denies complete axiomatizations.
It explicitly leaves it open whether, for instance, $0$ is a root of $3
\sqrt{|x|-1} \geq x+1$, while in mainstream mathematics the intention is that
it is not a root.
On the other hand,~\cite{GrS95} makes $0$ a root of $\frac{1}{x} - x =
\frac{1}{2x}$.

Both underspecification and negative (and positive) free logics support a
semiformal approach to dealing with undefined terms.
Under mild assumptions, there is a straightforward recursive algorithm that,
for each term $t$, produces a formula $\isdef{t}$ that yields $\TT$ when $t$
is defined, and $\FF$ otherwise~\cite{BB+05,DMR08,GrS95}.
For instance, $\frac{\sqrt{x}}{x-2}$ is defined precisely when $\sqrt{x}$ is
defined, $x-2$ is defined, and $x-2 \neq 0$, that is, $x \geq 0 \wedge x \neq
2$.

Let $\varphi(X)$ be a first-order formula containing precisely one instance
of the nullary relation symbol $X$, and not containing other connectives and
quantifiers than $\neg$, $\wedge$, $\vee$, $\forall$, and $\exists$.
Let $n$ be the number of those subformulas of $\varphi(X)$ that are of the
form $\neg \varphi'$, where $\varphi'$ contains $X$.
Let $R(t)$ be a formula, and let $\varphi(R(t))$ denote the result of putting
it in place of $X$.

If $n$ is even, then for any interpretation of all other non-logical symbols
than $X$, either $\varphi(\FF)$ yields the same truth value as $\varphi(\TT)$,
or $\varphi(\FF)$ yields $\FF$ and $\varphi(\TT)$ yields $\TT$.
In the former case, $\varphi(R(t))$, $\varphi(\isdef{t} \wedge R(t))$, and
$\varphi(\neg \isdef{t} \vee R(t))$ yield the same truth value.
In the latter case, $\varphi(\isdef{t} \wedge R(t))$ yields $\FF$ and
$\varphi(\neg \isdef{t} \vee R(t))$ yields $\TT$ when $t$ is undefined, and
both agree with $\varphi(R(t))$ when $t$ is defined.
The case that $n$ is odd can be returned to the even case by considering
$\varphi(\neg (\neg R(t)))$.
This makes it possible to choose the outcome of undefined terms as
appropriate, when translating the practical problem at hand to logical
formulas.

In the case of our running example, this approach asks someone to choose
\emph{informally} between $|x|-1 \geq 0 \wedge 3 \sqrt{|x|-1} \geq x+1$ and
$\neg(|x|-1 \geq 0) \vee 3 \sqrt{|x|-1} \geq x+1$, after which the student
should solve the chosen one formally.
We wish the choice be made by the student, on the basis that those values of
$x$ that make $\sqrt{|x|-1}$ undefined are not roots.
Furthermore, the reasoning behind the choice should be formalizable.

Many people share the intuition that an undefined formula is neither false nor
true, even if that results in the loss of the Law of Excluded Middle.
For instance,~\cite{Cha05} reports on a survey that was participated by over
200 software developers.
When asked how various undefined situations should be interpreted, between
74\,\% and 91\,\% chose ``error/exception'', the other options being
``true'', ``false'', and ``other (provide details)''.
One of the situations was a programming analogue of $\frac{1}{0} = 0 \vee
\frac{1}{0} \neq 0$.

So we turn our attention to 3-valued logics.
We denote the third truth value with $\UU$ and call it ``undefined''.
(Some authors talk of it as the absence of truth value instead of a truth
value.)
Everybody seems to agree that $\neg \UU$ yields $\UU$.
Now neither $3 \sqrt{|x|-1} \geq x+1$ nor $\neg(3 \sqrt{|x|-1} \geq x+1)$ has
$0$ as a root, because both of them yield $\UU$ when $x = 0$.
The definition of $t < t'$ as a shorthand for $\neg (t \geq t')$ works well.

In the presence of $\UU$, there are three major interpretations of $\wedge$
and $\vee$.
Fortunately, we need not elaborate them now, because we will see in
Section~\ref{S:Discuss} that the other two can be obtained from our choice in
Definition~\ref{D:sem-formula}(\ref{D:sem-and}) and~(\ref{D:sem-or}) as
shorthands.
Our choice is the same as Kleene's~\cite{Kle52} and
{\L}ukasiewicz's~\cite{Luk30}.
Our $\forall$ and $\exists$ are analogous.
The other versions of $\forall$ and $\exists$ that we encountered can be
mimicked by them.
What $\rightarrow$ should mean in this context is a tricky issue, as we will
argue in Section~\ref{S:Discuss}.
We will mimic Kleene's version in Section~\ref{S:Seman}, and {\L}ukasiewicz's
version in Section~\ref{S:Isdef}.

The important issue regarding related work is how to say that a term or
formula is undefined.
An obvious idea is to introduce a new atomic predicate ${*}t$ or connective
${*}\varphi$ that yields $\FF$ if its input is undefined and $\TT$
otherwise~\cite{BCJ84,deN17,GaL90,JoM94,McC63,Nol20,PaG21}.
Alternatively, the idea of $\isdef{t}$ discussed above can be used and
naturally extended to formulas~\cite{BB+05,DMR08}, at the cost of needing a
new component in addition to the traditional signature and set of axioms.
In the words of~\cite{BB+05}:
``we will assume that included with every signature $\Sigma$ is a set $\Delta$
of domain formulas, one for each function and predicate symbol in $\Sigma$.''
The difference is that ${*}$ is but $\isdef{}$ is not a symbol in the core
language.
Instead, $\isdef{t}$ and $\isdef{\varphi}$ are metalanguage expressions that
denote some formulas that typically only contain $\neg$, $\wedge$, $\vee$,
$\forall$, $\exists$, and atomic formulas.
We will adopt the latter approach in Section~\ref{S:Isdef}, and argue in
Section~\ref{S:Discuss} that it can mimic the former.

The replacement of $\varphi(R(t))$ by $\varphi(\isdef{t} \wedge R(t))$ or
$\varphi(\neg \isdef{t} \vee R(t))$ remains a powerful practical reasoning
method also in the presence of $\UU$, and then it takes place within the
formal logic.
More generally, assume that no other connectives and quantifiers are used than
(our versions of) $\neg$, $\wedge$, $\vee$, $\forall$, and $\exists$; $\psi$
is a formula; and $\varphi(\psi)$ is a formula where $\psi$ occurs within the
scope of an even number of negations.
Then every interpretation that makes $\varphi(\psi)$ yield $\TT$, also makes
$\varphi(\isdef{\psi} \wedge \psi)$ yield $\TT$, and vice versa.
A similar claim holds for odd number of negations and $\varphi(\neg
\isdef{\psi} \vee \psi)$.
The hard part in proving these is ruling out the possibility that $\psi$
yields $\UU$, $\varphi(\psi)$ yields $\TT$ and $\varphi(\isdef{\psi} \wedge
\psi)$ or $\varphi(\neg \isdef{\psi} \vee \psi)$ yields $\FF$ or $\UU$.
It becomes easy by appealing to the notion of \emph{regularity} proposed by
Kleene and developed in Section~\ref{S:Regul}.

In our running example, $\isdef{3 \sqrt{-x-1}}$ is $-x-1 \geq 0$, $\isdef{3
\sqrt{x-1}}$ is $x-1 \geq 0$, and there are no negations.
Therefore, $3 \sqrt{-x-1} \geq x+1$ can be replaced by $x \leq -1 \wedge 3
\sqrt{-x-1} \geq x+1$ and $3 \sqrt{x-1} \geq x+1$ by $x \geq 1 \wedge 3
\sqrt{x-1} \geq x+1$, resulting in a formula that yields $\UU$ for no value of
$x$.
In this way formulas can be formally converted to a form where undefinedness
plays essentially no role, after which classical binary logic can be used for
the rest of the reasoning (for further discussion and sources, please
see~\cite{DMR08}).
This approach can be used by both humans and computers.
It is simple to use, because $\isdef{t}$ and $\isdef{\varphi}$ are obtained
with an algorithm.
The algorithm will be presented in Defnition~\ref{D:isdef2}.

If $\varphi(\psi)$ is replaced by $\varphi(({*}\psi) \wedge \psi)$ instead of
$\varphi(\isdef{\psi} \wedge \psi)$, then the result contains  $*$, which is
not any of $\neg$, $\wedge$, $\vee$, $\forall$, and $\exists$.
As a consequence, doing a second replacement is not necessarily sound.
So the practical approach described above is lost.
On the other hand, if the original formula contains $*$, it can be replaced by
$\isdef{}$, opening the way to the practical approach described above.

In the present study we develop a logic that uses $\UU$ and $\isdef{}$,
present a proof system for it, and prove that the system is sound and
complete.
The number of rules in our proof system that differ from classical rules is
small.
We believe that this is the first completeness proof for a proof system that
relies on $\isdef{}$.
Furthermore, we believe to be the first to point out the role of regularity in
practice-oriented reasoning in this context (one example was above and another
will be in Section~\ref{S:Regul}).

Also~\cite{JoM94} claims completeness, but using $*$ and only for finite
axiomatizations.
The source~\cite{GaL90} presents a (in our opinion hard to read)
tableaux-based completeness proof for a logic that uses $*$.
Its notion of $\models$ is unusual in that both $\TT$ and $\UU$ are designated
values on the right (but only $\TT$ on the left).
As a consequence, its proof system would be only indirectly applicable to our
purposes.
The completeness claim in~\cite{DMR08} does not refer to Gödel's sense, but to
what in this study is Lemma~\ref{L:isdef}(\ref{L:id-sound1})
and~(\ref{L:id-sound2}):
``The procedure is complete [8,9], that is, the well-deﬁnedness condition
generated from a formula is provable if and only if the formula is
well-deﬁned.''

In terms of free logics, ours has \emph{neutral} semantics~\cite{Nol20}.
Positive or negative semantics only use the two truth values $\FF$ and $\TT$,
while neutral semantics and \emph{supervaluation} also use $\UU$.
A recent study \cite{PaG21} covers positive and negative semantics, but
leaves out the latter two, mentioning that they ``up to now still lack the
rigorous systematicity the other two family members enjoy''.
Supervaluation makes $\frac{1}{0} = \frac{1}{0}$ yield $\TT$ on the basis that
if $\frac{1}{0}$ is given any value, no matter what, then $\frac{1}{0} =
\frac{1}{0}$ would yield $\TT$ in classical logic.
In our logic $\frac{1}{0} = \frac{1}{0}$ yields $\UU$.
Neutral non-supervaluation semantics were surveyed in~\cite{Leh94}.
Our logic disagrees with all the systems summarized in the table on~p.~328.
For instance, unlike our logic, $\wedge$ and $\vee$ are \emph{strict}
in~\cite{Leh01}, that is, if $\varphi$ or $\psi$ or both are undefined, then
$\varphi \wedge \psi$ and $\varphi \vee \psi$ are undefined as well.

The studies~\cite{Neg02,Pet16,Win16} discuss proof systems for Kleene's logic,
but only cover propositional logic, while~\cite{ACN88} focuses on equational
logic without quantifiers.

The syntax and semantics of our core logic are presented in
Sections~\ref{S:Language} and~\ref{S:Seman}.
We already mentioned that the notion of regularity is developed in
Section~\ref{S:Regul}, and $\isdef{}$ in Section~\ref{S:Isdef}.
Section~\ref{S:Proof} is devoted to a proof system for our logic, together
with its soundness proof.
The system is proven complete (in the sense of Gödel, allowing infinite sets
of axioms) in Section~\ref{S:Compl}.
In Section~\ref{S:Discuss} we argue that most, if not all, other 3-valued
logics for similar applications can be mimicked by ours.

\section{Formal Languages}\label{S:Language}

Our notion of a \emph{formal language} is essentially the same as in classical
binary first-order logic.
A couple of details are affected by the needs of the rest of this study.
We will comment on them after presenting the definition.

The {alphabet of a formal language is the union of the following five mutually
disjoint sets:
\begin{enumerate}

\item The set $\mc{L}$ of the following eleven symbols: $($ $)$ $,$ $=$ $\FF$
$\TT$ $\neg$ $\wedge$ $\vee$ $\forall$ $\exists$

\item A countably infinite set $\mc{V}$ of \emph{variable symbols} $\var_1$,
$\var_2$, \ldots

\item A countable set $\mc{C}$ of \emph{constant symbols} $c_1$, $c_2$, \ldots

\item A countable set $\mc{F}$ of \emph{function symbols} $f_1$, $f_2$, \ldots

\item A countable set $\mc{R}$ of \emph{relation symbols} $R_1$, $R_2$, \ldots

\end{enumerate}

All formal languages have the same $\mc{L}$ and the same $\mc{V}$, but not
necessarily the same $\mc{C}$, $\mc{F}$, or $\mc{R}$.
In particular, we will assume that the variable symbols are literally
$\var_1$, $\var_2$, and so on.
To emphasize this, we write them as $\var_i$ instead of $v_i$.
When we want to refer to a variable symbol without saying which one, we use
$x$, $y$, $x_1$, and so on, as metalanguage variable symbols.

Each function symbol $f$ and each relation symbol $R$ has an \emph{arity}
$\arit(f)$ or $\arit(R)$.
It is a positive integer.
A \emph{signature} is the quadruple $(\mc{C}, \mc{F}, \mc{R}, \arit)$.

Terms, atomic formulas, and formulas are defined recursively as follows.
\begin{definition}\label{D:syntax}
Let a signature be fixed.
\begin{enumerate}

\item\label{D:syntax-t}
A \emph{term} is either a variable symbol; a constant symbol; or of the form\\
$\acall{f}{t}$, where $f$ is a function symbol and $t_1$, \ldots,
$t_{\arit(f)}$ are terms.

\item\label{D:syntax-a}
An \emph{atomic formula} is either $\FF$; $\TT$; of the form $(t_1 = t_2)$
where $t_1$ and $t_2$ are terms; or of the form $\acall{R}{t}$, where $R$ is a
relation symbol and $t_1$, \ldots, $t_{\arit(R)}$ are terms.

\item\label{D:syntax-f}
A \emph{formula} is either an atomic formula or of any of the following forms,
where $\varphi$ and $\psi$ are formulas and $x$ is a variable symbol:
$$(\neg \varphi) \mid (\varphi \wedge \psi) \mid (\varphi \vee \psi) \mid
(\forall x\s \varphi) \mid (\exists x\s \varphi)$$

\end{enumerate}
\end{definition}

An occurrence of a variable symbol $x$ in a formula is \emph{bound} if and
only if it is in a subformula of the form $(\forall x\s \varphi)$ or $(\exists
x\s \varphi)$.
The other occurrences of $x$ are \emph{free}.
A formula is \emph{closed} if and only if it has no free occurrences of
variable symbols, and \emph{open} in the opposite case.

We make a distinction between variables and variable symbols, for the reason
illustrated by the classical binary first-order logic formula $\exists
\var_2\s (\var_2 < \var_1 \wedge \exists \var_1\s (\var_1 < \var_2))$ on real
numbers.
If, for instance, the free occurrence of $\var_1$ has the value $3$, then the
formula can be shown to hold by letting $\var_2 = 2$ and the bound occurrences
of $\var_1$ have the value $1$.
Instead of thinking of the variable symbol $\var_1$ having simultaneously the
values $3$ and $1$, we think of the free and bound occurrences of $\var_1$ as
referring to two distinct variables which just happen to have the same name.
In general, there are two kinds of \emph{variables}: free and bound.
Every variable symbol that occurs free introduces a \emph{free variable}, and
every $\forall$ and every $\exists$ introduces a \emph{bound variable}.

We will use $\varphi(x)$ as a synonym for $\varphi$, and $\varphi(t)$ to
denote the result of replacing every free occurrence of $x$ in $\varphi$ by
$t$.
The purpose of the notation $\varphi(x)$ is to make it clear which is the
variable symbol whose free occurrences are replaced.
By \emph{$t$ is free for $x$ in $\varphi$} it is meant that no variable symbol
in $t$ becomes bound in $\varphi(t)$.
As is well known, replacing $x$ by $t$ that is not free for $x$ is often
incorrect, because, intuitively speaking, a free variable in $t$ disappears
and another, bound variable with the same name takes its place.

The atomic formulas $\FF$ and $\TT$, corresponding to the truth values false
and true, have been included in the language for technical convenience.
Although we will also talk about a third truth value $\UU$ (undefined), we did
not include a corresponding atomic formula in the language.
Thanks to this design choice, our theory reduces to classical binary
first-order logic when every function symbol is defined everywhere.

We will tell how to add Kleene's versions of the symbols $\rightarrow$ and
$\leftrightarrow$ to the language in Section~\ref{S:Seman}, and
{\L}ukasiewicz's version of $\rightarrow$ in Section~\ref{S:Isdef}.

We adopt some semiformal conventions to improve the readability of formulas.
To reduce the need of $($ and $)$, we let $\neg$ have the highest precedence,
then $\wedge$, then $\vee$, and finally the quantifiers $\forall$ and
$\exists$.
The connectives $\wedge$ and $\vee$ associate to the left; that is, $\varphi
\wedge \psi \wedge \chi$ denotes $((\varphi \wedge \psi) \wedge \chi)$, and
similarly for $\vee$.
In examples we may use the familiar syntax of the domain of discourse of the
example.
For instance, in the case of natural numbers we may write
$$\forall n\s (\neg \exists m\s (m \cdot m = n)) \vee \sqrt{n} \cdot \sqrt{n}
= n$$
as a human-friendly semiformal representation of the formula
$$(\forall \var_1\s ((\neg (\exists \var_2\s (\cdot(\var_2, \var_2) =
\var_1))) \vee (\cdot(\sqrt{}(\var_1), \sqrt{}(\var_1)) = \var_1)))$$

Let $\varphi$, $\psi$, and $\chi$ denote any formulas.
In the metalanguage, we use $\varphi \samef \psi$ to denote that after
unwinding all semiformal abbreviations, $\varphi$ and $\psi$ result in
literally the same formula.
For instance, because we have chosen that both $\varphi \vee \psi \vee \chi$
and $(\varphi \vee \psi) \vee \chi$ are abbreviations for $((\varphi \vee
\psi) \vee \chi)$, we have $\varphi \vee \psi
\vee \chi \samef (\varphi \vee \psi) \vee \chi$.
On the other hand, we have $\varphi \vee \psi \vee \chi \not\samef \varphi
\vee (\psi \vee \chi)$, because $(\varphi \vee (\psi \vee \chi))$ is not
literally the same formula as $((\varphi \vee \psi) \vee \chi)$.

\section{Structures, Truth Values, and Models}\label{S:Seman}

Our notion of a structure differs from the standard one in that function
symbols may denote partial functions.
More formally, a \emph{structure} $(\dom, \itpr)$ on a signature $(\mc{C},
\mc{F}, \mc{R}, \arit)$ consists of the following:
\begin{enumerate}

\item
A non-empty set $\dom$, called the \emph{domain of discourse}.

\item
For each $c \in \mc{C}$, an element $\tulk{C}$ of $\dom$.
(We reserve the symbol $\tulk{c}$ for another use.)

\item
For each $f \in \mc{F}$, a partial function $\tulk{f}$ from $\dom^{\arit(f)}$
to $\dom$.

\item
For each $R \in \mc{R}$, a subset $\tulk{R}$ of $\dom^{\arit(R)}$.

\end{enumerate}

The standard approach would continue by defining an assignment of values for
variable symbols, and then defining a value for each term and a truth value
for each formula.
We will proceed in the opposite order, by first interpreting terms as partial
functions and formulas as total functions, and then assigning values to
variable symbols.
We do so to simplify the formulation of the notion of ``regularity'' that will
be presented in Definition~\ref{D:reg}.

In itself, reversing the order is insignificant, since it does not affect the
fundamental ideas but only their formalization.
However, because of the overall goal of our study, we also introduce two
significant changes.
First, terms need not yield a value.
Second, the values of formulas are picked from among three truth values
\emph{false}, \emph{undefined} and \emph{true}, denoted by $\FF$, $\UU$ and
$\TT$, respectively.

To interpret terms and formulas as functions, we need to map variable symbols
to argument positions.
For instance, we need to decide whether $\var_3+\var_2$ is interpreted as the
function $\dom^2 \to \dom; (d_3, d_2) \mapsto d_3+d_2$ where $d_i$ denotes the
value of $\var_i$, or as something else.
We interpret it as $\dom^3 \to \dom; (d_1, d_2, d_3) \mapsto d_3+d_2$.
In general, we let the arguments correspond to $\var_1$, $\var_2$, and so on,
in this order, as far as needed by the term or formula.
This will make many argument lists contain positions whose corresponding
variable symbol does not occur (free) in the term or formula, such as the
first position and $\var_1$ in $\dom^3 \to \dom; (d_1, d_2, d_3) \mapsto
d_3+d_2$.
This will not be much of a problem, because interpretation as functions is
only an auxiliary tool.
After assigning values to free variables, a standard kind of interpretation is
obtained.

For each term $t$ we define its arity $\arit(t)$ as the biggest $i$ such that
$\var_i$ occurs in $t$.
If no variable symbol occurs in $t$, then $\arit(t) = 0$.
It is easy to see that the arity of a compound term $\acall{f}{t}$ is the
maximum of the arities of its constituents $t_1$, \ldots, $t_{\arit(f)}$.
Similarly, we define that the arity of a closed formula is $0$, and the arity
of an open formula is the biggest $i$ such that $\var_i$ occurs free in it.
It is possible that $\arit((\forall x\s \varphi)) < \arit(\varphi)$ and
$\arit((\exists x\s \varphi)) < \arit(\varphi)$.
Even so, the arity of a compound formula is at most the maximum of the arities
of its constituents.

\begin{definition}\label{D:sem-term}
Given a signature $(\mc{C}, \mc{F}, \mc{R}, \arit)$ and a structure $(\dom,
\itpr)$ on it, each term $t$ defines a partial function $\tulk{t}$ from
$\dom^{\arit(t)}$ to $\dom$ as follows.
In the definition, $d_1$, \ldots, $d_n$ are arbitrary elements of $\dom$.
\begin{enumerate}

\item\label{D:sem-v}
If $\var_n \in \mc{V}$, then $\tulk{\var_n}$ is the total function from
$\dom^n$ to $\dom$ that maps $(d_1, \ldots, d_n)$ to $d_n$.
That is, $\tulk{\var_n}(d_1, \ldots, d_n) = d_n$.

\item\label{D:sem-c}
If $c \in \mc{C}$, then $\tulk{c}$ is the function with arity zero such that
$\tulk{c}() = \tulk{C}$.
That is, $\tulk{c}: \dom^0 \to \dom; () \mapsto \tulk{C}$.

\item\label{D:sem-f}
If $f \in \mc{F}$ and $t_1$, \ldots, $t_{\arit(f)}$ are terms, then let $n =
\max\{\arit(t_1), \ldots, \arit(t_{\arit(f)})\}$.
We define $\tulk{\acall{f}{t}}$ as the following partial function from
$\dom^n$ to $\dom$.
\begin{enumerate}

\item[--]
If for $1 \leq i \leq \arit(f)$, any of the $\tcall{t_i}{d}$ is undefined,
then\\
$\tulk{\acall{f}{t}}(d_1, \ldots, d_n)$ is undefined as well.

\item[--]
Otherwise, for $1 \leq i \leq \arit(f)$ let $e_i = \tcall{t_i}{d}$.
If $\tcall{f}{e}$ is defined, then $\tulk{\acall{f}{t}}(d_1, \ldots, d_n) =
\tcall{f}{e}$; and otherwise $\tulk{\acall{f}{t}}(d_1, \ldots, d_n)$ is
undefined.

\end{enumerate}

\end{enumerate}
\end{definition}

The definition obeys the principle that if any subterm of a term is undefined,
then the term as a whole is undefined as well.
That is, our partial functions are \emph{strict}.

\begin{definition}\label{D:sem-formula}
Given a signature and a structure $(\dom, \itpr)$ on it, each formula
$\varphi$ defines a total function $\tulk{\varphi}$ from
$\dom^{\arit(\varphi)}$ to $\FUTset$ as follows.
In the definition, $d_1$, \ldots, $d_n$ are arbitrary elements of $\dom$.
To avoid confusion with the formal symbol $=$, we write $\tcall{\varphi}{d}
\yield \FF$ to denote that $\tulk{\varphi}$ maps $\param{\varphi}{d}$ to
$\FF$, and similarly with $\UU$ and $\TT$.
\begin{enumerate}

\item\label{D:sem-F-T}
We define that $\tulk{\FF}$ and $\tulk{\TT}$ are the functions with arity zero
whose values are $\FF$ and $\TT$, respectively.

\item\label{D:sem-=}
Let $n = \max\{\arit(t_1), \arit(t_2)\}$.
We define $\tulk{(t_1 = t_2)}(d_1, \ldots, d_n)$ $\yield$
\begin{enumerate}

\item[$\UU$,] if $\tcall{t_1}{d}$ or $\tcall{t_2}{d}$ is undefined

\item[$\TT$,] if $\tcall{t_1}{d}$ and $\tcall{t_2}{d}$ are defined, and
$$\tcall{t_1}{d} = \tcall{t_2}{d}$$

\item[$\FF$,] if $\tcall{t_1}{d}$ and $\tcall{t_2}{d}$ are defined, and
$$\tcall{t_1}{d} \neq \tcall{t_2}{d}$$

\end{enumerate}

\item\label{D:sem-R}
Let $n = \max\{\arit(t_1), \ldots, \arit(t_{\arit(R)})\}$, and let $e_i =
\tcall{t_i}{d}$ when the latter is defined.
We define $\tulk{\acall{R}{t}}(d_1, \ldots, d_n)$ $\yield$
\begin{enumerate}

\item[$\UU$,] if for $1 \leq i \leq \arit(R)$, any of the $\tcall{t_i}{d}$ is
undefined

\item[$\TT$,] if for $1 \leq i \leq \arit(R)$, each $\tcall{t_i}{d}$ is
defined and $\param{R}{e} \in \tulk{R}$

\item[$\FF$,] if for $1 \leq i \leq \arit(R)$, each $\tcall{t_i}{d}$ is
defined and $\param{R}{e} \notin \tulk{R}$

\end{enumerate}

\item\label{D:sem-neg}
Clearly $\arit((\neg \varphi)) = \arit(\varphi)$.
We define $\tcall{(\neg \varphi)}{d}$ $\yield$
\begin{enumerate}

\item[$\FF$,]
if $\tcall{\varphi}{d} \yield \TT$

\item[$\TT$,]
if $\tcall{\varphi}{d} \yield \FF$

\item[$\UU$,]
if $\tcall{\varphi}{d} \yield \UU$

\end{enumerate}

\item\label{D:sem-and}
Let $n = \max\{\arit(\varphi), \arit(\psi)\}$.
We define $\tulk{(\varphi \wedge \psi)}(d_1, \ldots, d_n)$ $\yield$
\begin{enumerate}

\item[$\TT$,] if $\tcall{\varphi}{d} \yield \tcall{\psi}{d} \yield \TT$

\item[$\FF$,] if $\tcall{\varphi}{d} \yield \FF$ or $\tcall{\psi}{d} \yield
\FF$

\item[$\UU$,] otherwise

\end{enumerate}

\item\label{D:sem-or}
Let $n = \max\{\arit(\varphi), \arit(\psi)\}$.
We define $\tulk{(\varphi \vee \psi)}(d_1, \ldots, d_n)$ $\yield$
\begin{enumerate}

\item[$\FF$,] if $\tcall{\varphi}{d} \yield \tcall{\psi}{d} \yield \FF$

\item[$\TT$,] if $\tcall{\varphi}{d} \yield \TT$ or $\tcall{\psi}{d} \yield
\TT$

\item[$\UU$,] otherwise

\end{enumerate}

\item\label{D:sem-forall}
If $i > \arit(\varphi)$, we define $\tulk{(\forall \var_i\s \varphi)}$ as
$\tulk{\varphi}$.
Otherwise $1 \leq i \leq \arit(\varphi)$, and we define $\tcall{(\forall
\var_i\s \varphi)}{d}$ $\yield$
\begin{enumerate}

\item[$\TT$,] if for every $e_i \in \dom$ we have $\tulk{\varphi}(d_1, \ldots,
e_i, \ldots, d_{\arit(\varphi)}) \yield \TT$

\item[$\FF$,] if for at least one $e_i \in \dom$ we have $\tulk{\varphi}(d_1,
\ldots, e_i, \ldots, d_{\arit(\varphi)}) \yield \FF$

\item[$\UU$,] otherwise

\end{enumerate}

\item\label{D:sem-exists}
If $i > \arit(\varphi)$, we define $\tulk{(\exists \var_i\s \varphi)}$ as
$\tulk{\varphi}$.
Otherwise $1 \leq i \leq \arit(\varphi)$, and we define $\tcall{(\exists
\var_i\s \varphi)}{d}$ $\yield$
\begin{enumerate}

\item[$\FF$,] if for every $e_i \in \dom$ we have $\tulk{\varphi}(d_1, \ldots,
e_i, \ldots, d_{\arit(\varphi)}) \yield \FF$

\item[$\TT$,] if for at least one $e_i \in \dom$ we have $\tulk{\varphi}(d_1,
\ldots, e_i, \ldots, d_{\arit(\varphi)}) \yield \TT$

\item[$\UU$,] otherwise

\end{enumerate}

\end{enumerate}
\end{definition}
It follows from~(\ref{D:sem-=}) that if both $t_1$ and $t_2$ are defined, then
$t_1 = t_2$ compares their values in the usual fashion, and if at least one of
them is undefined, then $(t_1 = t_2) \yield \UU$.
For instance, with real numbers, $\sqrt{-1} = \sqrt{-1}$ is not true but
undefined.

More generally, by~(\ref{D:sem-F-T}), (\ref{D:sem-=}), and~(\ref{D:sem-R}), an
atomic formula yields $\UU$ only if it contains an undefined term.
This restriction is only for technical convenience.
It may be circumvented by introducing a new function symbol $f$ that is
undefined precisely when desired, making $\param{R}{d} \in \tulk{R}$ when
$\tcall{f}{d}$ is undefined, and using $\acall{R}{x} \wedge (\acall{f}{x} =
\acall{f}{x})$.

\begin{figure}
\begin{center}%\mbox{}\hfill
\begin{tabular}{c|c}
$\neg$ &\\
\hline
$\FF$ & $\TT$\\
$\UU$ & $\UU$\\
$\TT$ & $\FF$
\end{tabular}\hfill
\begin{tabular}{c|c@{~}c@{~}c}
$\wedge$ & $\FF$ & $\UU$ & $\TT$\\
\hline
$\FF$ & $\FF$ & $\FF$ & $\FF$\\
$\UU$ & $\FF$ & $\UU$ & $\UU$\\
$\TT$ & $\FF$ & $\UU$ & $\TT$
\end{tabular}\hfill
\begin{tabular}{c|c@{~}c@{~}c}
$\vee$ & $\FF$ & $\UU$ & $\TT$\\
\hline
$\FF$ & $\FF$ & $\UU$ & $\TT$\\
$\UU$ & $\UU$ & $\UU$ & $\TT$\\
$\TT$ & $\TT$ & $\TT$ & $\TT$
\end{tabular}\hfill
\begin{tabular}{c|c@{~}c@{~}c}
$\rightarrow$ & $\FF$ & $\UU$ & $\TT$\\
\hline
$\FF$ & $\TT$ & $\TT$ & $\TT$\\
$\UU$ & $\UU$ & $\UU$ & $\TT$\\
$\TT$ & $\FF$ & $\UU$ & $\TT$
\end{tabular}\hfill
\begin{tabular}{c|c@{~}c@{~}c}
$\leftrightarrow$ & $\FF$ & $\UU$ & $\TT$\\
\hline
$\FF$ & $\TT$ & $\UU$ & $\FF$\\
$\UU$ & $\UU$ & $\UU$ & $\UU$\\
$\TT$ & $\FF$ & $\UU$ & $\TT$
\end{tabular}\hfill
\begin{tabular}{c|c@{~}c@{~}c}
$\lukimpl$ & $\FF$ & $\UU$ & $\TT$\\
\hline
$\FF$ & $\TT$ & $\TT$ & $\TT$\\
$\UU$ & $\UU$ & $\TT$ & $\TT$\\
$\TT$ & $\FF$ & $\UU$ & $\TT$
\end{tabular}%\hfill\mbox{}
\end{center}
\caption{Truth tables of some propositional connectives.
The symbols $\rightarrow$ and $\leftrightarrow$ are Kleene's conditional and
biconditional, and $\lukimpl$ is {\L}ukasiewicz's
conditional}\label{F:truthtables}
\end{figure}

It is easy to check that~(\ref{D:sem-neg}), (\ref{D:sem-and}),
and~(\ref{D:sem-or}) make $\neg$, $\wedge$, and $\vee$ match the corresponding
truth tables in Figure~\ref{F:truthtables}.
Furthermore, Kleene's conditional and biconditional can be obtained by
treating $\varphi \rightarrow \psi$ as a shorthand for $\neg \varphi \vee
\psi$, and $\varphi \leftrightarrow \psi$ as a shorthand for $(\varphi
\rightarrow \psi) \wedge (\psi \rightarrow \varphi)$.
We will show in Section~\ref{S:Regul} that {\L}ukasiewicz's conditional
$\lukimpl$ cannot be expressed in our language.
However, Section~\ref{S:Isdef} will reveal that any formula that contains it
can be replaced by a formula in our language.

In~(\ref{D:sem-forall}) and~(\ref{D:sem-exists}), if $i > \arit(\varphi)$, it
is appropriate to define quantification so that it has no effect, because then
$\var_i$ does not occur free in $\varphi$.
The definitions for the case $i \leq \arit(\varphi)$ are analogous to the
definitions of $\wedge$ and $\vee$.
In them, $\var_i$ may but need not occur free in $\varphi$.

If every function symbol in a formula $\varphi$ is defined everywhere, then
everywhere $\tcall{\varphi}{d} \not\yield \UU$.
If all function symbols of the language are defined everywhere, then
Definitions~\ref{D:sem-term} and~\ref{D:sem-formula} reduce to the classical
binary first-order logic semantics represented in a function form.

It is helpful to think of $\FF$ being smaller than $\UU$ which is smaller than
$\TT$.
Then $\tcall{(\varphi \wedge \psi)}{d}$ yields the minimum of the results of
$\tcall{\varphi}{d}$ and $\tcall{\psi}{d}$, and $\tcall{(\varphi \vee
\psi)}{d}$ yields the maximum.
Furthermore, $\tcall{(\forall \var_i\s \varphi(\var_i))}{d}$ yields the
minimum of $\tulk{\varphi}'(e)$ for $e \in \dom$, and $\tcall{(\exists
\var_i\s \varphi(\var_i))}{d}$ yields the maximum.
Here $\tulk{\varphi}'$ denotes the function from $\dom$ to $\FUTset$ obtained
by using $d_1$, \ldots, $d_{\arit(\varphi)}$ as other arguments of
$\tulk{\varphi}$ than the $i$th.
That is, if $i > \arit(\varphi)$, then $e \mapsto \tcall{\varphi}{d}$, and
otherwise $e \mapsto \tulk{\varphi}(d_1, \ldots, e, \ldots,
d_{\arit(\varphi)})$.

It is easy to check that De Morgan's laws hold in our logic:
\begin{lemma}\label{L:DeMorgan}
\mbox{}
\begin{enumerate}

\item $\tulk{(\neg(\varphi \wedge \psi))}$ is the same function as
$\tulk{((\neg\varphi) \vee (\neg\psi))}$.

\item $\tulk{(\neg(\forall x\s \varphi))}$ is the same function as
$\tulk{(\exists x\s (\neg\varphi))}$.

\end{enumerate}
\end{lemma}

If a term is not free for a variable symbol in a formula, then the following
lemma can be used to change the names of the bound variables in the formula,
so that the term becomes free.

\begin{lemma}\label{L:var_change}
If $y$ does not occur in $\varphi(x)$, then $\tulk{(\forall y\s \varphi(y))}$
is the same function as $\tulk{(\forall x\s \varphi(x))}$, and $\tulk{(\exists
y\s \varphi(y))}$ is the same function as $\tulk{(\exists x\s \varphi(x))}$.
\end{lemma}
\begin{proof}
If $x$ and $y$ are the same variable symbol, or if $x$ does not occur free in
$\varphi(x)$, then the claim is trivial.
So we assume that they are distinct and $x$ does occur free.

By construction, $x$ does not occur free in $\varphi(y)$.
By assumption, $y$ does not occur free in $\varphi(x)$.
Therefore, $\forall x\s \varphi(x)$ and $\forall y\s \varphi(y)$ have the same
free variables.
Let $n = \arit((\forall x\s \varphi(x))) = \arit((\forall y\s \varphi(y)))$.

Because $y$ does not occur in $\varphi(x)$, all occurrences of $y$ in
$\varphi(y)$ are free, and they match precisely the free occurrences of $x$ in
$\varphi(x)$.
Let $i$ and $j$ be such that $x$ is $\var_i$ and $y$ is $\var_j$.
The functions $\tulk{\varphi(x)}$ and $\tulk{\varphi(y)}$ have $\max\{n, i\}$
and $\max\{n,j\}$ arguments, respectively, but their values only depend on
those arguments whose corresponding variable occurs free.
The values of $x$ and $y$ go in via different argument positions, but are from
then on treated identically.
The values of all other free variables are treated fully identically.

Therefore, $\tulk{\varphi(x)}(d_1, \ldots, d_n; e_i) \yield
\tulk{\varphi(y)}(d_1, \ldots, d_n; e_j)$, where the notation has the
following meaning.
The symbols $d_1$, \ldots, $d_{\max\{n, i-1, j-1\}}$, and $e_i$ denote
arbitrary elements of $\dom$, and $e_j = e_i$.
If $i \leq n$, then $(d_1, \ldots, d_n; e_i)$ denotes $(d_1, \ldots, e_i,
\ldots, d_n)$, and if $i > n$, it denotes $(d_1, \ldots, d_{i-1}, e_i)$.
Furthermore, $(d_1, \ldots, d_n; e_j)$ is defined similarly.

As a consequence, $\tulk{(\forall x\s \varphi(x))}(d_1, \ldots, d_n) \yield
\tulk{(\forall y\s \varphi(y))}(d_1, \ldots, d_n)$, and similarly with
$\exists$.
\end{proof}

\begin{definition}\label{D:nu}
Let a signature be fixed.

Given a structure $\sigma = (\dom, \itpr)$ on it, an \emph{assignment of
values to free variables} is a total function $\nu$ from $\mathbb{Z}^+$ to
$\dom$.
Given $\sigma$ and $\nu$, any term $t$ yields $\tulk{t}(\nu(1), \ldots,
\nu(\arit(t)))$, and any formula $\varphi$ yields $\tulk{\varphi}(\nu(1),
\ldots, \nu(\arit(\varphi)))$.

A \emph{model} of a formula $\varphi$ is a pair $(\sigma, \nu)$ such that
$\tulk{\varphi}(\nu(1), \ldots, \nu(\arit(\varphi))) \yield \TT$.
This is denoted with $(\sigma, \nu) \models \varphi$.
If $\Gamma$ is a set of formulas, then $(\sigma, \nu) \models \Gamma$ means
that for every $\varphi \in \Gamma$ we have $(\sigma, \nu) \models \varphi$.
\end{definition}

If $x$ denotes the variable $\var_i$, then by $\nu(x)$ we mean $\nu(i)$.
Let $d \in \dom$.
By $\nu[x := d]$ we denote the assignment such that $\nu[x := d](x) = d$ and
$\nu[x := d](y) = \nu(y)$ when $y$ is not the same variable as $x$.
For brevity, we will often write $\tulk{t}(\nu)$ instead of $\tulk{t}(\nu(1),
\ldots, \nu(\arit(t)))$ and $\tulk{\varphi}(\nu)$ instead of
$\tulk{\varphi}(\nu(1), \ldots, \nu(\arit(\varphi)))$.
The following lemma is immediate from Definition~\ref{D:sem-term}.

\begin{lemma}\label{L:t-d}
Assume that $\tulk{t}(\nu)$ is defined and yields the value $d$.
If $t$ is free for $x$ in $\varphi(x)$, then $\tulk{\varphi(t)}(\nu) \yield
\tulk{\varphi(x)}(\nu[x := d])$.
\end{lemma}

\section{Regularity}\label{S:Regul}

In this section we introduce and discuss a notion that needs different
technical background from the rest of this study.
We first briefly introduce the necessary background.

By a \emph{3-valued propositional logic} we mean a logic whose alphabet
consists of $\FF$, $\UU$, $\TT$, proposition symbols, and a choice of
propositional connectives.
Any formula in the logic whose proposition symbols are among $P_1$, \ldots,
$P_n$ can be interpreted as a truth function from $\FUTset^n$ to $\FUTset$.
Kleene's 3-valued propositional logic~\cite{Kle52,Pet16} has the connectives
$\neg$, $\wedge$, $\vee$, $\rightarrow$, and $\leftrightarrow$ introduced in
Figure~\ref{F:truthtables}.
{\L}ukasiewicz's 3-valued propositional logic~\cite{Luk30} has $\neg$,
$\wedge$, $\vee$, $\lukimpl$, and a biconditional version of $\lukimpl$.

One can check from Figure~\ref{F:truthtables} that the truth functions
represented by $P \vee Q$, $P \rightarrow Q$, and $P \leftrightarrow Q$ can
also be represented as $\neg(\neg P \wedge \neg Q)$, $\neg P \vee Q$, and $(P
\rightarrow Q) \wedge (Q \rightarrow P)$, respectively.
On the other hand, we will soon see that $P \lukimpl Q$ cannot be constructed
from other connectives in the figure.

Kleene's 3-valued propositional logic has a useful property called
\emph{regularity}.
Intuitively, it says that if the truth value of a formula depends on the truth
value of $P_i$ (while the truth values of the other proposition symbols remain
unchanged), then the truth value of the formula is $\UU$ when the truth value
of $P_i$ is $\UU$.

\begin{definition}\label{D:reg_Prop}
Let $\pi(P_1, \ldots, P_n)$ be a truth function.
It is \emph{regular} if and only if for each $1 \leq i \leq n$, for each $j$
such that $1 \leq j \leq n$ and $j \neq i$, and for each $P_j \in \FUTset$
\begin{enumerate}

\item
either $\pi(P_1, \ldots, \UU, \ldots, P_n) \yield \UU$

\item
or $\pi(P_1, \ldots, \FF, \ldots, P_n) \yield \pi(P_1, \ldots, \UU, \ldots,
P_n) \yield \pi(P_1, \ldots, \TT, \ldots, P_n)$,

\end{enumerate}
where the explicitly shown truth value $\FF$, $\UU$, or $\TT$, is assigned to
$P_i$.

A formula is regular if and only if the truth function represented by it is
regular.
A propositional logic is regular if and only if all of its formulas are
regular.
\end{definition}

It is easy to see from Figure~\ref{F:truthtables} that $P \lukimpl Q$ is not
regular: $\UU \lukimpl \UU \yield \TT \not\yield \UU$, but $\TT \lukimpl \UU
\yield \UU \not\yield \UU \lukimpl \UU$.
Therefore, {\L}ukasiewicz's 3-valued propositional logic is not regular.

In Figure~\ref{F:truthtables}, excluding $\lukimpl$, each row and each column
either has $\UU$ in the middle, or its every entry is $\FF$ or every entry is
$\TT$.
Therefore, $\neg P$, $P \wedge Q$, $P \vee Q$, $P \rightarrow Q$, and $P
\leftrightarrow Q$ are regular.
It is possible to prove (and we will de facto do so as part of the proof of
Theorem~\ref{T:reg}) that every propositional formula that is composed only
using proposition symbols, $\FF$, $\TT$, $\UU$, $\neg$, $\wedge$, $\vee$,
$\rightarrow$, and $\leftrightarrow$ is regular.
As a consequence, Kleene's propositional logic is regular.
Therefore, $P \lukimpl Q$ cannot be constructed in it.
It is also impossible to construct a formula $*(P)$ such that $*(\UU) \yield
\FF$ and $*(\TT) \yield *(\FF) \yield \TT$, because it is irregular.
(On the other hand, $*(P) \yield \neg((P \lukimpl \neg P) \wedge (\neg P
\lukimpl P))$, and $P \lukimpl Q \yield \neg P \vee Q \vee \neg(*(P) \vee
*(Q))$.)

Next we adapt the notion of regularity to our predicate logic.
Although the value of a variable is never undefined, it is possible to assign
an undefined term in the place of each free occurrence of the variable symbol.
We will need handy notation for discussing such situations.
Therefore, we introduce a new metalanguage symbol $\bot$, to be used only in
this section, to represent the missing value of an undefined term.

We declare $\bot \notin \dom$ and define $\domb = \dom \cup \{\bot\}$.
Then we extend each $\tulk{t}$ to a partial function $\tulkb{t}$ from
$\domb^{\arit(t)}$ to $\dom$, and each $\tulk{\varphi}$ to a total function
$\tulkb{\varphi}$ from $\domb^{\arit(\varphi)}$ to $\FUTset$.
(We do not follow the well-known approach of extending $\tulk{t}$ to a total
function from $\domb^{\arit(t)}$ to $\domb$, because we want to use $\bot$ as
little as possible.)
The desired effect is obtained by rewriting Definition~\ref{D:sem-term}
and~\ref{D:sem-formula} such that $\tulkb{~~}$ is used instead of $\itpr$, and
\ref{D:sem-term}(\ref{D:sem-v}) is replaced by the following:
\begin{quote}
If $\var_n \in \mc{V}$, then $\tulkb{\var_n}$ is the partial function from
$\domb^n$ to $\dom$ such that if $e_n \in \dom$, then $\tulkb{\var_n}(e_1,
\ldots, e_n) = e_n$, and otherwise $\tulkb{\var_n}(e_1, \ldots, e_n)$ is
undefined.
\end{quote}

By an ``extended value'' of a free variable $\var_i$ we mean an element of
$\domb$ as the $i$th argument of $\tulkb{t}$ or $\tulkb{\varphi}$.
Intuitively, regularity says that for any free variable $\var_i$, if the truth
value of a formula depends on the extended value of $\var_i$ (while the
extended values of the other free variables remain unchanged), then the truth
value of the formula is $\UU$ when the extended value of $\var_i$ is
undefined.

\begin{definition}\label{D:reg}
Let $\varphi$ be a formula and $n = \arit(\varphi)$.
The formula $\varphi$ is \emph{regular} if and only if for each $1 \leq i \leq
n$, for each $j$ such that $1 \leq j \leq n$ and $j \neq i$, and for each $e_1
\in \domb$, \ldots, $e_{i-1} \in \domb$, $e_{i+1} \in \domb$, \ldots, $e_n \in
\domb$
\begin{enumerate}

\item
either $\tulkb{\varphi}(e_1, \ldots, \bot, \ldots, e_n) \yield \UU$

\item
or $\tulkb{\varphi}(e_1, \ldots, e_i, \ldots, e_n) \yield \tulkb{\varphi}(e_1,
\ldots, \bot, \ldots, e_n)$ for every $e_i \in \dom$,

\end{enumerate}
where the value $\bot$ or $e_i$ is used as the $i$th argument.
\end{definition}

The following theorem is from~\cite{VaH17}, but we have improved its proof.
\begin{theorem}\label{T:reg}
The logic in Definitions~\ref{D:syntax}, \ref{D:sem-term}, and
\ref{D:sem-formula} is regular.
\end{theorem}
\begin{proof}
Let $\varphi$, $n$, $i$, and $e_1$, \ldots, $e_n$ be like in
Definition~\ref{D:reg}.
For brevity, if $\psi$ is any subformula of $\varphi$, we write
$\tulkb{\psi}(e_i)$ instead of $\tcallb{\psi}{e}$ both when $1 \leq i \leq
\arit(\psi)$ and when $i > \arit(\psi)$.
We use induction on the structure of $\varphi$ to show that
$\tulkb{\varphi}(\bot) \yield \UU$ or $\tulkb{\varphi}(e_i)$ is the same for
every $e_i \in \domb$.

The base case consists of atomic formulas.
By Definition~\ref{D:sem-formula}(\ref{D:sem-F-T}), $\tulkb{\FF}(e_i) \yield
\FF$ and $\tulkb{\TT}(e_i) \yield \TT$ independently of $e_i$.
If $\var_i$ occurs in neither $t_1$ nor $t_2$, then $\tulkb{(t_1 = t_2)}(e_i)$
does not depend on $e_i$.
Otherwise, if $e_i$ is $\bot$, then $t_1$ or $t_2$ is undefined, so by
\ref{D:sem-formula}(\ref{D:sem-=}) $\tulkb{(t_1 = t_2)}(\bot) \yield \UU$.
By~\ref{D:sem-formula}(\ref{D:sem-R}), similar reasoning applies to
$\acall{R}{t}$.
So the atomic formulas are regular.

The induction step consists of five cases.
By the induction assumption, the subformula(s) $\psi$, $\psi_1$, and $\psi_2$
of each case are regular.

Let $\varphi$ be $\neg \psi$.
By \ref{D:sem-formula}(\ref{D:sem-neg}), if $\tulkb{\psi}(\bot) \yield \UU$,
then also $\tulkb{\varphi}(\bot) \yield \UU$.
If $\tulkb{\psi}(e_i) \yield \FF$ independently of $e_i$, then
$\tulkb{\varphi}(e_i) \yield \TT$ independently of $e_i$.
If $\tulkb{\psi}(e_i) \yield \TT$ independently of $e_i$, then
$\tulkb{\varphi}(e_i) \yield \FF$ independently of $e_i$.

Let $\varphi$ be $\psi_1 \wedge \psi_2$.
By \ref{D:sem-formula}(\ref{D:sem-and}), if $\tulkb{\psi_1}(e_i) \yield
\tulkb{\psi_2}(e_i) \yield \TT$ independently of $e_i$, then also
$\tulkb{\varphi}(e_i) \yield \TT$ independently of $e_i$.
If $\tulkb{\psi_1}(e_i) \yield \FF$ or $\tulkb{\psi_2}(e_i) \yield \FF$
independently of $e_i$, then also $\tulkb{\varphi}(e_i) \yield \FF$
independently of $e_i$.
In the remaining cases $\tulkb{\psi_1}(\bot)$ $\not\yield$ $\FF$ $\not\yield$
$\tulkb{\psi_2}(\bot)$, and $\tulkb{\psi_1}(\bot) \yield \UU$ or
$\tulkb{\psi_2}(\bot) \yield \UU$.
Then $\tulkb{\varphi}(\bot) \yield \UU$.

Let $\varphi$ be $\forall x\s \psi$.
We write $\tulkb{\psi}(e_i; d)$ to indicate that the value of $x$ is $d \in
\dom$.
By \ref{D:sem-formula}(\ref{D:sem-forall}), if for every $d \in \dom$ we have
$\tulkb{\psi}(e_i; d) \yield \TT$ independently of $e_i$, then also
$\tulkb{\varphi}(e_i) \yield \TT$ independently of $e_i$.
If for some $d \in \dom$ we have $\tulkb{\psi}(e_i; d) \yield \FF$
independently of $e_i$, then also $\tulkb{\varphi}(e_i) \yield \FF$
independently of $e_i$.
In the remaining cases at least one $d \in \dom$ yields $\tulkb{\psi}(\bot; d)
\yield \UU$, and no $d \in \dom$ yields $\tulkb{\psi}(\bot; d) \yield \FF$.
Then $\tulkb{\varphi}(\bot) \yield \UU$.

The cases $\psi_1 \vee \psi_2$ and $\exists x\s \psi$ are proven similarly
using \ref{D:sem-formula}(\ref{D:sem-or}) and
\ref{D:sem-formula}(\ref{D:sem-exists}).
\end{proof}

We gave an example in Section~\ref{S:Intro} that regularity is important in
practical application of our logic.
Just to give another example that can be explained briefly: let $t$ and $t'$
be terms such that they are free for $x$ in $\varphi(x)$, and when $t$ is
defined, then $t = t'$.
For instance, we may have $t = \frac{x-2}{x-2}$ and $t' = 1$.
If the logic is regular, then $\varphi(t)$ implies $\varphi(t')$.
This is because when $\tulk{\varphi(t)} \yield \TT$ but $t$ is undefined, then
$\tulk{\varphi(t')} \yield \TT$ by regularity.
This makes it correct to solve $\frac{x-2}{x-2}(x^2-5x+7) = 1$ by replacing
$\frac{x-2}{x-2}$ by $1$, solving $1(x^2-5x+7) = 1$, and checking its roots
$2$ and $3$ against the original equation.
The root $2$ fails and $3$ passes the check, so $3$ is the only root of
$\frac{x-2}{x-2}(x^2-5x+7) = 1$.

From now on we will not use $\bot$ explicitly.
Instead, when we appeal to regularity in the sequel, we will use the following
corollary of Theorem~\ref{T:reg}.
\begin{corollary}\label{C:reg}
If $\tulk{t}(\nu)$ is undefined but $\tulk{\varphi(t)}(\nu) \yield \TT$ or
$\tulk{\varphi(t)}(\nu) \yield \FF$, then for every $d \in \dom$ we have
$\tulk{\varphi(x)}(\nu[x := d]) \yield \tulk{\varphi(t)}(\nu)$.
\end{corollary}
\begin{proof}
In the notation of Definition~\ref{D:reg}, $\tulk{\varphi(t)}(\nu)$ is
$\tulkb{\varphi}(\nu(1), \ldots, \bot, \ldots, \nu(n))$ and
$\tulk{\varphi(x)}(\nu[x := d])$ is $\tulkb{\varphi}(\nu(1), \ldots, d,
\ldots, \nu(n))$.
\end{proof}

\section{Is Defined -Formulas}\label{S:Isdef}

To understand the motivation of the topic of this section, consider adding a
function symbol for multiplicative inverses to the theory of real closed
fields.
The standard axiom $\forall x\s (x = 0 \vee x \cdot \frac{1}{x} = 1)$ does
most of the job.
When $x=0$, then $(x \cdot \frac{1}{x} = 1) \yield \UU$, but $(x = 0 \vee x
\cdot \frac{1}{x} = 1) \yield \TT$ by
Definition~\ref{D:sem-formula}(\ref{D:sem-or}).
What this axiom fails to do is to tell that $\frac{1}{0}$ has been
intentionally left undefined.
It leaves open many possibilities, including $\frac{1}{0} = 0$ and
$\frac{1}{0} = 1$.
It thus leaves the axiomatization incomplete.

In everyday mathematics it is natural to use a first-order formula to specify
the domain of a function.
For instance, in the case of real numbers, $\frac{y}{x}$ is defined precisely
when $x \neq 0$; $\sqrt{x}$ is defined precisely when $x \geq 0$; and $\log x$
is defined precisely when $x > 0$.

To formalize this idea, we assume that in a formal theory, each function
symbol $f$ has an associated \emph{isdef-formula} $\isdef{f}$, defined soon.
The name is an abbreviation of ``is defined -formula''.
While in classical binary first-order logic a theory consists of two
components: a signature and a set of formulas on it (the axioms), in our logic
a theory consists of three components: the signature, the axioms, and the
isdef-formulas.
The isdef-formulas will be defined so that they never yield $\UU$.
We use $\idf$ to denote the mapping from the function symbols to their
isdef-formulas, and $\isdef{f}$ denotes the image of $f \in \mc{F}$.
The notation $\isdef{}$ used in Section~\ref{S:Intro} applies the idea to
terms and formulas.
It will be defined in terms of $\idf$ in Definition~\ref{D:isdef2}.

As was discussed in more detail towards the end of Section~\ref{S:Language},
$\varphi \samef \psi$ means that $\varphi$ and $\psi$ denote the literally
same formula.

\begin{definition}\label{D:isdef1}
\emph{Isdef-formulas} on a signature $(\mc{C}, \mc{F}, \mc{R}, \arit)$ are a
function $\idf$ from $\mc{F}$ to the formulas on the signature such that for
every $f \in \mc{F}$,
\begin{enumerate}

\item\label{D:id-var}
$\isdef{f}$ contains no other free variables than $\var_1$, \ldots,
$\var_{\arit(f)}$, and

\item\label{D:id-g}
for every function symbol $g$ in $\isdef{f}$, we have $\isdef{g} \samef \TT$.

\end{enumerate}
\end{definition}

To improve readability, in examples we may write $x$, $y$, and $z$ instead of
$\var_1$, $\var_2$, and $\var_3$.
Here are some examples on familiar function symbols on real numbers:
\begin{center}
$\isdef{x+y} \samef \TT$, $\isdef{\sqrt{x}} \samef (x \geq 0)$, and
$\isdef{\frac{x}{y}} \samef (\neg(y = 0))$.
\end{center}
Because $\TT$, $(x \geq 0)$, and $(\neg(y = 0))$ contain no function symbols
at all, they vacuously satisfy Definition~\ref{D:isdef1}(\ref{D:id-g}).
Because multiplication is defined on all pairs of natural numbers, we may
choose $\isdef{x \cdot y} \samef \TT$.
Then an isdef-formula of the square root on natural numbers could be $\exists
y\s (y \cdot y = x)$.
It contains the function symbol $\cdot$.

The intention is that each function is defined precisely when its
isdef-formula yields $\TT$.
The next definition expresses this property, and tells how isdef-formulas are
taken into account in the notions of model and logical
consequence.
The notation $(\sigma, \nu) \models \varphi$ and $(\sigma, \nu) \models
\Gamma$ was introduced in Definition~\ref{D:nu}.

\begin{definition}\label{D:id-models}
Let $\Gamma$ be a set of formulas and $\idf$ be the isdef-formulas.
\begin{enumerate}

\item\label{D:id-f}
A \emph{model of $\idf$} is a structure $(\dom, \itpr)$ such that for every
function symbol $f$ and every $d_1 \in \dom$, \ldots, $d_{\arit(f)} \in \dom$,
the following holds:
\begin{center}
$\icall{f}{d} \yield \TT$ if and only if $\tulk{f}(d_1, \ldots, d_{\arit(f)})$
is defined.
\end{center}
This is denoted by $(\dom, \itpr) \models \idf$.

\item\label{D:id-m}
A \emph{model of $(\idf, \Gamma)$} is a pair $(\sigma, \nu)$ such that
$\sigma$ is a structure, $\nu$ is an assignment of values to free variables,
$\sigma \models \idf$, and $(\sigma, \nu) \models \Gamma$.

\item\label{D:id-|=}
A formula $\varphi$ is a \emph{logical consequence} of $\idf$ and $\Gamma$,
denoted by $(\idf, \Gamma) \models \varphi$, or $\Gamma \models \varphi$ for
brevity, if and only if every model of $(\idf, \Gamma)$ also is a model of
$\varphi$.

\end{enumerate}
\end{definition}

To illustrate the contribution of $\idf$ to the notion of logical consequence,
let $\arit(f) = 1$, $\isdef{f}_1 \samef \TT$, $\isdef{f}_2 \samef \FF$, and
$\varphi \samef (f(x) = f(x))$.
We have $(\idf_1, \emptyset) \models \varphi$ but $(\idf_2, \emptyset)
\not\models \varphi$.

In this study, given a signature, $\Gamma$ will vary frequently, but $\idf$
will remain the same.
As a consequence, $\Gamma$ is informative but $\idf$ is dead weight in the
notation $(\idf, \Gamma) \models \varphi$.
Therefore, we prefer the notation $\Gamma \models \varphi$, but remind at
places that the concept also depends on $\idf$.

\begin{definition}\label{D:theory}
A \emph{3-valued first-order theory} is a triple $(\mc{S}, \Gamma, \idf)$,
where:
\begin{enumerate}
\item $\mc{S}$ is a signature,
\item $\Gamma$ is a set of formulas on $\mc{S}$ (known as the \emph{axioms}),
and
\item $\idf$ is isdef-formulas on $\mc{S}$.
\end{enumerate}
\end{definition}

For instance, the function $\frac{1}{x}$ can be added to the classical binary
first-order theory of real closed fields as follows.
First, the isdef-formula $\TT$ is introduced for each original function
symbol, to make the theory 3-valued.
Then, the symbol $\frac{1}{}$ is added to the signature; the formula $\neg(x =
0)$ is made its isdef-formula; and the formula $\forall x\s (x = 0 \vee x
\cdot \frac{1}{x} = 1)$ is added to the axioms.
The square root function can be added to the theory of natural numbers by
introducing the isdef-formulas $\TT$, adding $\sqrt{\phantom{x}}$ to the
signature, giving it the isdef-formula $\exists y\s (y \cdot y = x)$, and
adding the axiom $(\neg \exists y\s (y \cdot y = x)) \vee (\sqrt{x} \cdot
\sqrt{x} = x)$.

It is intuitively clear that, for instance, in the case of real numbers,
$\frac{\sqrt{x+y}}{y+1}$ is defined if and only if $x+y \geq 0 \wedge y+1 \neq
0$.
Next we present and show correct a straightforward algorithm that implements
this intuition.
Given isdef-formulas, for each term $t$ it computes a formula $\isdef{t}$ and
for each formula $\varphi$ it computes a formula $\isdef{\varphi}$.
Given a structure that models the isdef-formulas, $\isdef{t}$ yields $\TT$ if
and only if $t$ is defined, and $\isdef{\varphi}$ yields $\TT$ if and only if
$\varphi$ does not yield $\UU$.
The formulas $\isdef{t}$ and $\isdef{\varphi}$ themselves never yield $\UU$.
\begin{definition}\label{D:isdef2}
Assume that a signature and isdef-formulas on it are given.
\begin{enumerate}

\item\label{D:id-vc}
If $t$ is a variable symbol or a constant symbol, then $\isdef{t} \samef \TT$.

\item\label{D:id-t}
$\isdef{\acall{f}{t}} \samef \isdef{t_1} \wedge \cdots \wedge
\isdef{t_{\arit(f)}} \wedge \icall{f}{t}$\\
(Where necessary, rename bound variables in $\isdef{f}$ as justified by
Lemma~\ref{L:var_change}, so that $t_1$, \ldots, $t_{\arit(\isdef{f})}$ become
free for $\var_1$, \ldots, $\var_{\arit(\isdef{f})}$.)

\item $\isdef{\FF} \samef \isdef{\TT} \samef \TT$\label{D:id-F-T}

\item $\isdef{t_1 = t_2} \samef \isdef{t_1} \wedge \isdef{t_2}$\label{D:id-=}

\item $\isdef{\acall{R}{t}} \samef \isdef{t_1} \wedge \cdots \wedge
\isdef{t_{\arit(R)}}$\label{D:id-R}

\item $\isdef{\neg \varphi} \samef \isdef{\varphi}$\label{D:id-neg}

\item $\isdef{\varphi \wedge \psi} \samef (\isdef{\varphi} \wedge
\isdef{\psi}) \vee (\isdef{\varphi} \wedge \neg \varphi) \vee (\isdef{\psi}
\wedge \neg \psi)$\label{D:id-and}

\item $\isdef{\varphi \vee \psi} \samef (\isdef{\varphi} \wedge \isdef{\psi})
\vee (\isdef{\varphi} \wedge \varphi) \vee (\isdef{\psi} \wedge
\psi$)\label{D:id-or}

\item $\isdef{\forall x\s \varphi} \samef (\forall x\s \isdef{\varphi}) \vee
\exists x\s (\isdef{\varphi} \wedge \neg \varphi)$\label{D:id-forall}

\item $\isdef{\exists x\s \varphi} \samef (\forall x\s \isdef{\varphi}) \vee
\exists x\s (\isdef{\varphi} \wedge \varphi)$\label{D:id-exists}

\end{enumerate}
\end{definition}

\begin{lemma}\label{L:isdef}
Assume $(\dom, \itpr) \models \idf$.
Let $t$ be a term and $\varphi$ be a formula.
\begin{enumerate}

\item\label{L:id-no-new-var}
Every free variable in $\isdef{t}$ also occurs in $t$, and every free variable
in $\isdef{\varphi}$ also occurs free in $\varphi$.

\item\label{L:id-2ary}
If for every function symbol $f$ and every $d_1 \in \dom$, \ldots,
$d_{\arit(f)} \in \dom$ we have $\icall{f}{d} \yield \TT$, then for every term
$t$, every formula $\varphi$, and every $d_1 \in \dom$, $\ldots$ we have
$\icall{t}{d} \yield \TT$ and $\icall{\varphi}{d} \yield \TT$.

\item\label{L:id-sound1}
Let $d_1 \in \dom$, \ldots, $d_{\arit(t)} \in \dom$.
If $\tcall{t}{d}$ is defined, then $\icall{t}{d}$ $\yield$ $\TT$.
Otherwise $\icall{t}{d} \yield \FF$.

\item\label{L:id-sound2}
Let $d_1 \in \dom$, \ldots, $d_{\arit(\varphi)} \in \dom$.
If $\tcall{\varphi}{d} \yield \FF$ or $\tcall{\varphi}{d} \yield \TT$, then
$\icall{\varphi}{d} \yield \TT$.
Otherwise $\icall{\varphi}{d} \yield \FF$.

\item\label{L:id-recurs}
Assume that $\idf$ is a computable function.
Then the functions $t \mapsto \isdef{t}$ and $\varphi \mapsto \isdef{\varphi}$
are computable.

\end{enumerate}

\end{lemma}
\begin{proof}\mbox{}
\begin{enumerate}

\item
It is easy to check from Definition~\ref{D:isdef2} that no case introduces
other free variables than those in $t_1$, \ldots, $t_n$, $\isdef{t_1}$,
\ldots, $\isdef{t_n}$, $\varphi$, $\psi$, $\isdef{\varphi}$, and
$\isdef{\psi}$, where $t_1$, \ldots, $t_n$ are the proper subterms and
$\varphi$ and $\psi$ are the proper subformulas of the case.
The claim follows from this immediately by induction.

\item
Each case in Definition~\ref{D:isdef2} is $\TT$, $\isdef{t_1} \wedge \cdots
\wedge \isdef{t_{\arit(f)}} \wedge \icall{f}{t}$, $\isdef{t_1} \wedge \cdots
\wedge \isdef{t_n}$ for some $n$, or $\isdef{\varphi}$; or contains the
disjunct $\isdef{\varphi} \wedge \isdef{\psi}$ or $\forall x\s
\isdef{\varphi}$.
By induction and the assumption of the claim, each such formula or disjunct,
and thus each case, yields $\TT$ everywhere.

\item
We use induction on the structure of $t$.
For brevity we drop parameter lists of the form $\param{~}{d}$, but do show
those that are of other forms.

The base case consists of variable and constant symbols.
By Definition~\ref{D:sem-term}(\ref{D:sem-v}) and~(\ref{D:sem-c}),
$\tulk{\var_i}$ and $\tulk{c}$ are defined, and by
Definition~\ref{D:isdef2}(\ref{D:id-vc}) $\tulk{\isdef{\var_i}}$ $\yield$
$\tulk{\isdef{c}} \yield \TT$.

The induction step consists of terms of the form $\acall{f}{t}$.
By the induction hypothesis and Definition~\ref{D:id-models}(\ref{D:id-f}),
there are three cases, each of which can be dealt with
\ref{D:sem-term}(\ref{D:sem-f}) and \ref{D:isdef2}(\ref{D:id-t}):
\begin{enumerate}

\item
At least one of the $\tulk{t_i}$ is undefined and $\tulk{\isdef{t_i}} \yield
\FF$.
Then $\tulk{\acall{f}{t}}$ is undefined and $\tulk{\isdef{\acall{f}{t}}}
\yield \FF$.

\item
Every $\tulk{t_i}$ is defined and has $\tulk{\isdef{t_i}} \yield \TT$; and
$\tcall{f}{d'}$ is defined, where each $d'_i$ is the value of $\tulk{t_i}$.
By \ref{D:id-models}(\ref{D:id-f}) we have $\icall{f}{d'}$ $\yield$ $\TT$.
Then $\tulk{\acall{f}{t}}$ is defined and $\tulk{\isdef{\acall{f}{t}}} \yield
\TT$.

\item
Every $\tulk{t_i}$ is defined and has $\tulk{\isdef{t_i}} \yield \TT$; but
$\tcall{f}{d'}$ is undefined, where each $d'_i$ is the value of $\tulk{t_i}$.
By \ref{D:id-models}(\ref{D:id-f}) we have $\icall{f}{d'}$ $\yield$ $\FF$.
Then $\tulk{\acall{f}{t}}$ is undefined and $\tulk{\isdef{\acall{f}{t}}}
\yield \FF$.

\end{enumerate}

\item
We use induction on the structure of $\varphi$, and again drop parameter lists
of the form $\param{~}{d}$.
The base case consists of atomic formulas.
\begin{enumerate}

\item
By Definition~\ref{D:sem-formula}(\ref{D:sem-F-T}) $\tulk{\FF} \yield \FF$ and
$\tulk{\TT} \yield \TT$, which matches
Definition~\ref{D:isdef2}(\ref{D:id-F-T}).

\item
If $\tulk{t_1}$ or $\tulk{t_2}$ is undefined, then
by~\ref{D:sem-formula}(\ref{D:sem-=}) $\tulk{t_1 = t_2} \yield \UU$.
By Claim~\ref{L:id-sound1}, $\tulk{\isdef{t_1}} \yield \FF$ or
$\tulk{\isdef{t_2}} \yield \FF$, so by~\ref{D:isdef2}(\ref{D:id-=})
$\tulk{\isdef{t_1 = t_2}} \yield \FF$.
Otherwise $\tulk{t_1}$ and $\tulk{t_2}$ are defined.
Then by~\ref{D:sem-formula}(\ref{D:sem-=}) $\tulk{t_1 = t_2} \yield \FF$ or
$\tulk{t_1 = t_2} \yield \TT$.
By Claim~\ref{L:id-sound1}, $\tulk{\isdef{t_1}} \yield \tulk{\isdef{t_2}}
\yield \TT$, so by~\ref{D:isdef2}(\ref{D:id-=}) $\tulk{\isdef{t_1 = t_2}}
\yield \TT$.

\item
The case $\tcall{R}{t}$ is proven similarly to (b)
using~\ref{D:sem-formula}(\ref{D:sem-R}) and~\ref{D:isdef2}(\ref{D:id-R}).

\end{enumerate}

The induction step consists of five cases.
\begin{enumerate}

\item
By~\ref{D:sem-formula}(\ref{D:sem-neg}), $\tulk{\neg \varphi} \yield \UU$ if
and only if $\tulk{\varphi} \yield \UU$.
This matches~\ref{D:isdef2}(\ref{D:id-neg}).

\item
If $\tulk{\varphi} \yield \tulk{\psi} \yield \TT$, then by
\ref{D:sem-formula}(\ref{D:sem-and}) $\tulk{\varphi \wedge \psi} \yield \TT$.
By the induction assumption and \ref{D:isdef2}(\ref{D:id-and}),
$\tulk{\isdef{\varphi}}$, $\tulk{\isdef{\psi}}$, and $\tulk{\isdef{\varphi
\wedge \psi}}$ yield $\TT$.
If $\tulk{\varphi} \yield \FF$, then $\tulk{\varphi \wedge \psi} \yield \FF$,
and $\tulk{\isdef{\varphi}}$, $\tulk{(\isdef{\varphi} \wedge \neg \varphi)}$,
and $\tulk{\isdef{\varphi \wedge \psi}}$ yield $\TT$.
Similarly if $\tulk{\psi} \yield \FF$, then $\tulk{\varphi \wedge \psi} \yield
\FF$ and $\tulk{\isdef{\varphi \wedge \psi}} \yield \TT$.
In the remaining cases, at least one of $\tulk{\varphi}$ and $\tulk{\psi}$
yields $\UU$ while the other yields $\TT$ or $\UU$, $\tulk{\varphi \wedge
\psi} \yield \UU$, at least one of $\tulk{\isdef{\varphi}}$ and
$\tulk{\isdef{\psi}}$ yields $\FF$, and $\tulk{\isdef{\varphi \wedge \psi}}
\yield \FF$.

\item
The case $\varphi \vee \psi$ is proven similarly to (b), but using
\ref{D:sem-formula}(\ref{D:sem-or}) and \ref{D:isdef2}(\ref{D:id-or}), and
with $\FF$ and $\TT$ swapped except when yielded by $\isdef{}$.

\item
To deal with $\forall x\s \varphi$, let $\tulk{\varphi}(;e)$ abbreviate
$\tcall{\varphi}{d}$, if $x$ is none of $\var_1$, \ldots,
$\var_{\arit(\varphi)}$; and otherwise $\tulk{\varphi}(;e)$ abbreviates
$\tulk{\varphi}(d_1, \ldots, e, \ldots, d_{\arit(\varphi)})$, where $e$ is
used in the place of $x$.
If for every $e \in \dom$ we have $\tulk{\varphi}(;e) \yield \TT$, then by
\ref{D:sem-formula}(\ref{D:sem-forall}) $\tulk{\forall x\s \varphi} \yield
\TT$.
By the induction assumption, for every $e \in \dom$ we have
$\tulk{\isdef{\varphi}}(;e) \yield \TT$.
That is, $\tulk{\forall x\s \isdef{\varphi}} \yield \TT$.
By \ref{D:isdef2}(\ref{D:id-forall}), $\tulk{\isdef{\forall x\s \varphi}}
\yield \TT$.
If $e \in \dom$ is such that $\tulk{\varphi}(;e) \yield \FF$, then
$\tulk{\forall x\s \varphi} \yield \FF$ and $\tulk{\isdef{\varphi}}(;e) \yield
\TT$.
So $\tulk{\exists x\s (\isdef{\varphi} \wedge \neg \varphi)} \yield \TT$ and
$\tulk{\isdef{\forall x\s \varphi}} \yield \TT$.
The case remains where $\tulk{\varphi}(;e) \yield \UU$ for at least one $e \in
\dom$ and $\tulk{\varphi}(;e) \yield \FF$ for no $e \in \dom$.
Then $\tulk{\forall x\s \varphi} \yield \UU$ and $\tulk{\isdef{\forall x\s
\varphi}} \yield \FF$.

\item
The case $\exists x\s \varphi$ is proven similarly to (d), but using
\ref{D:sem-formula}(\ref{D:sem-exists}) and \ref{D:isdef2}(\ref{D:id-exists}),
and with $\FF$ and $\TT$ swapped except when yielded by $\isdef{}$.

\end{enumerate}

\item
The algorithm is immediate from Definition~\ref{D:isdef2}, except perhaps the
renaming of bound variables in (\ref{D:id-t}).
A simple possibility starts by scanning $t_1$, \ldots, $t_{\arit(f)}$, to find
the greatest $i$ such that $i = 0$ or $\var_i$ occurs in at least one of them.
Then it replaces the bound variables in $\isdef{f}$ by $\var_{i+1}$,
$\var_{i+2}$, and so on.

\end{enumerate}
\end{proof}

For example, $\Isdef{\frac{\sqrt{x+y}}{y+1}}$
$$\begin{array}{rl}
\samef & \isdef{\sqrt{x+y}} \wedge \isdef{y+1} \wedge (\neg(y+1 = 0))\\
\samef & (\isdef{x+y} \wedge (x+y \geq 0)) \wedge (\isdef{y} \wedge \isdef{1}
  \wedge \TT) \wedge (\neg(y+1 = 0))\\
\samef & ((\isdef{x} \wedge \isdef{y} \wedge \TT) \wedge (x+y \geq 0)) \wedge
  (\TT \wedge \TT \wedge \TT) \wedge (\neg(y+1 = 0))\\
\samef & ((\TT \wedge \TT \wedge \TT) \wedge (x+y \geq 0)) \wedge (\TT \wedge
  \TT \wedge \TT) \wedge (\neg(y+1 = 0))\\
\end{array}$$
By Definition~\ref{D:sem-formula}(\ref{D:sem-F-T}) and~(\ref{D:sem-and}) and
the usual definition of $\neq$, this formula expresses the same function as
$x+y \geq 0 \wedge y+1 \neq 0$.

We emphasize that there is no symbol $\isdef{}$ in the formal language.
Every instance of $\isdef{}$ is a metalanguage expression that represents the
formal language expression that is obtained by Definition~\ref{D:isdef2}.

Now $\varphi \lukimpl \psi$ can be introduced as a shorthand for $\neg \varphi
\vee \psi \vee \neg(\isdef{\varphi} \vee \isdef{\psi})$.
This is not in contradiction with the facts that our logic is regular and
$\lukimpl$ cannot be expressed in a regular logic, because $\isdef{\varphi}$
only exists in the metalanguage.
It does not represent a truth function but a function from formulas to
formulas.

\section{Proof System and Its Soundness}\label{S:Proof}

Our proof system is loosely based on a proof system for classical binary
first-order logic in \cite[Section 3.1]{Duf91}.
The most important (but not only) difference is that our system depends on the
isdef-formulas $\isdef{f}$ of the function symbols $f \in \mc{F}$.
Therefore, its soundness proof will use Lemma
\ref{L:isdef}(\ref{L:id-sound1}) and (\ref{L:id-sound2}).
Thanks to \ref{L:isdef}(\ref{L:id-recurs}), if the set of axioms is recursive
and the function $\idf$ is computable, then also the proof system is
recursive.
If $\mc{F}$ is finite, then it is trivial that $\idf$ is computable.

The notation $(\idf, \Gamma) \proves \varphi$ means that $\varphi$ can be
proven from $\idf$ and $\Gamma$ using the proof system.
We usually write it more briefly as $\Gamma \proves \varphi$, just like we do
with $\models$, because $\idf$ never changes during our argumentation.
The meaning of $\Gamma \models \varphi$ was given in
Definition~\ref{D:id-models}(\ref{D:id-|=}).
Given a signature and the isdef-formulas, a proof system is \emph{sound} if
and only if for every $\Gamma$ and $\varphi$, $\Gamma \proves \varphi$ implies
$\Gamma \models \varphi$.

In what follows, $\varphi$, $\psi$ and $\chi$ are arbitrary formulas, $\Gamma$
and $\Delta$ are arbitrary sets of formulas, $t$, $t_1$, \ldots, $t_n$ are
arbitrary terms, and $x$ and $y$ are arbitrary variable symbols.
The rule schemas that differ from or are absent in classical binary
first-order logic have been marked with (*) at the end of the first line of
the rule schema.
After each group of rule schemas, we show their soundness if it is not
immediately obvious.
The soundness proof is by induction, where the induction assumption says that
every $\proves$ in the if-part of the schema is sound.

The first three rule schemas allow thinking of proofs as maintaining a set of
axioms and proven formulas, which grows each time a new formula is proven.
\begin{description}

\item[P1] $\{\varphi\} \proves \varphi$

\item[P2] If $\Gamma \proves \varphi$ then $\Gamma \cup \Delta \proves
\varphi$.

\item[P3] If $\Gamma \proves \varphi$ and $\Gamma \cup \{\varphi\} \proves
\psi$, then $\Gamma \proves \psi$.

\end{description}
\begin{lemma}
P1, P2, and P3 are sound.
\end{lemma}
\begin{proof}\mbox{}
\begin{itemize}

\item[P1]
For every $\sigma$ and $\nu$ such that $\sigma \models \idf$ and $(\sigma,
\nu) \models \{\varphi\}$, obviously $(\sigma, \nu) \models \varphi$.
This means that by Definition~\ref{D:id-models}(\ref{D:id-|=}), P1 is sound.

\item[P2]
If $(\sigma, \nu) \models \Gamma \cup \Delta$, then clearly $(\sigma, \nu)
\models \Gamma$.
If, furthermore, $\sigma \models \idf$ and $\Gamma \proves \varphi$, then by
the induction assumption $(\sigma, \nu) \models \varphi$, showing that P2 is
sound.

\item[P3]
If $\sigma \models \idf$, $(\sigma, \nu) \models \Gamma$, and $\Gamma \proves
\varphi$, then $(\sigma, \nu) \models \Gamma \cup \{\varphi\}$.
Then $\Gamma \cup \{\varphi\} \proves \psi$ yields $(\sigma, \nu) \models
\psi$.
Therefore, P3 is sound.

\end{itemize}
\end{proof}

The next rule schema expresses the Law of Excluded Fourth, which replaces the
Law of Excluded Middle in classical logic.
The next two are the basis of proof by contradiction.
The fourth one says that if a formula is true, then it is also defined.
\begin{description}

\item[C1] $\emptyset \proves (\varphi \vee \neg \varphi) \vee \neg
\isdef{\varphi}$\hfill(*)

\item[C2] $\{\FF\} \proves \varphi$

\item[C3] $\{\varphi, \neg \varphi\} \proves \FF$

\item[D1] $\{\varphi\} \proves \isdef{\varphi}$\hfill(*)

\end{description}
\begin{lemma}
C1, C2, C3, and D1 are sound.
\end{lemma}
\begin{proof}
Assume $\sigma \models \idf$.
By Lemma~\ref{L:isdef}(\ref{L:id-sound2}), if $\tulk{\varphi}(\nu)$ yields
neither $\TT$ nor $\FF$, then $\tulk{\isdef{\varphi}}(\nu) \yield \FF$.
So C1 is sound.
The soundness of D1 follows immediately from the same lemma.
The rule schemas C2 and C3 are vacuously sound, because there are no $\sigma$
and $\nu$ such that $(\sigma, \nu) \models \FF$ or $(\sigma, \nu) \models
\{\varphi, \neg \varphi\}$.
\end{proof}

Here is an example of a rule instance generated by D1:
$$\Big\{\frac{\sqrt{x}}{x-1} > 0\Big\} \proves \big(\TT \wedge (x \geq 0)\big)
\wedge \big(\TT \wedge \TT \wedge \TT\big) \wedge \big(\neg(x-1 = 0)\big)$$
Here is an example of a rule instance generated by C1, made more
human-readable by dropping a number of parentheses and ``$\TT \wedge {}$'':
$$\displaystyle \emptyset \proves \Big(\frac{\sqrt{x}}{x-1} > 0\Big) \vee
\neg\Big(\frac{\sqrt{x}}{x-1} > 0\Big) \vee \neg\Big((x \geq 0) \wedge
\neg(x-1 = 0)\Big)$$

We now illustrate the use of some rule schemas introduced this far, and obtain
two results that are needed later.
C4 is actually a theorem schema and its proof is a proof schema.
They become a theorem and a proof in our system by putting a formula in the
place of $\varphi$.
Also C5 is a theorem schema.
Our proof of it is not a proof schema but a demonstration that the set of
rules that C5 generates is a subset of those generated by C4.

\begin{lemma}\label{L:C4-C5}
For every formula $\varphi$, our proof system proves the following:
\begin{description}

\item[C4] $\{\varphi, \neg \isdef{\varphi}\} \proves \FF$

\item[C5] $\{\neg \varphi, \neg \isdef{\varphi}\} \proves \FF$

\end{description}
\end{lemma}
\begin{proof}\mbox{}
\begin{enumerate}

\item[C4]
By D1, $\{\varphi\} \proves \isdef{\varphi}$.
By P2, $\{\varphi, \neg \isdef{\varphi}\} \proves \isdef{\varphi}$.
Let this be called (1).
Using $\isdef{\varphi}$ in the place of $\varphi$ in C3, we get
$\{\isdef{\varphi}, \neg \isdef{\varphi}\} \proves \FF$.
By P2, $\{\varphi, \neg \isdef{\varphi}, \isdef{\varphi}\}$ $\proves$ $\FF$.
Let this be called (2).
By (1), (2) and P3, $\{\varphi, \neg \isdef{\varphi}\} \proves \FF$.

\item[C5]
With $\neg \varphi$ in the place of $\varphi$, C4 yields $\{\neg \varphi, \neg
\isdef{\neg \varphi}\}$ $\proves$ $\FF$.
By Definition~\ref{D:isdef2}(\ref{D:id-neg}) $\isdef{\neg \varphi}$ is
literally the same formula as $\isdef{\varphi}$.
Therefore, $\{\neg \varphi, \neg \isdef{\neg \varphi}\} \proves \FF$ is
literally the same as $\{\neg \varphi, \neg \isdef{\varphi}\} \proves \FF$.

\end{enumerate}
\end{proof}

It is clear that the above style of proof is clumsy indeed.
Therefore, we now argue that a finite sequence of subproofs and sets
$\Gamma_0$, \ldots, $\Gamma_n$ can also be thought of as a proof, where $n >
0$, $\Gamma_0 = \Gamma$, $\Gamma_i = \Gamma_{i-1} \cup \{\varphi_i\}$, and
$\varphi_i$ is obtained from some subset of $\Gamma_{i-1}$ by some rule
schema or subproof.
We do that by verifying by induction that if $0 \leq k \leq n$, then
$\Gamma_{n-k} \proves \varphi_n$.
The claim then follows by choosing $k = n$.

By P1 $\{\varphi_n\} \proves \varphi_n$.
So by P2, $\Gamma_{n-0} \proves \varphi_n$.
Thus the claim holds for $k=0$.
By the induction hypothesis $\Gamma_{n-k} \proves \varphi_n$.
By P2 and the assumption on how $\varphi_i$ is obtained, $\Gamma_{(n-k)-1}
\proves \varphi_{n-k}$.
Since $\Gamma_{n-k} = \Gamma_{n-k-1} \cup \{\varphi_{n-k}\}$, P3 yields
$\Gamma_{n-(k+1)} \proves \varphi_n$.

By this observation, proofs can be presented in a compressed form $\Gamma$
$\prv{x$_1$}$ $\varphi_1$ $\prv{x$_2$}$ $\cdots$ $\prv{x$_n$}$ $\varphi_n$,
where x$_1$, \ldots, x$_n$ are labels of rule schemas or lemmas, and
$\varphi_i$ can be obtained by x$_i$ from some subset of $\Gamma \cup
\{\varphi_1, \ldots, \varphi_{i-1}\}$.
Furthermore, more information can be added to the index x$_i$, to tell about
how the rule schema or lemma is applied.
In this representation, the proof of C4 may be written as $\{\varphi, \neg
\isdef{\varphi}\}$ $\prv{D1}$ $\isdef{\varphi}$ $\prv{C3}$ $\FF$.

The following rule schemas for conjunction and disjunction are mostly trivial.
The schema $\vee$-E expresses the principle of proof by cases.
\begin{description}

\item[$\wedge$-I] $\{\varphi, \psi\} \proves \varphi \wedge \psi$

\item[$\wedge$-E1] $\{\varphi \wedge \psi\} \proves \varphi$

\item[$\wedge$-E2] $\{\varphi \wedge \psi\} \proves \psi$

\item[$\vee$-I1] $\{\varphi\} \proves \varphi \vee \psi$

\item[$\vee$-I2] $\{\psi\} \proves \varphi \vee \psi$

\item[$\vee$-E] If $\Gamma \cup \{\varphi\} \proves \chi$ and $\Gamma \cup
\{\psi\} \proves \chi$, then $\Gamma \cup \{\varphi \vee \psi\} \proves \chi$.

\end{description}
\begin{lemma}
$\wedge$-I, $\wedge$-E1, $\wedge$-E2, $\vee$-I1, $\vee$-I2, and $\vee$-E are
sound.
\end{lemma}
\begin{proof}
If $\sigma \models \idf$ and $(\sigma, \nu) \models \Gamma \cup \{\varphi \vee
\psi\}$, then $(\sigma, \nu) \models \Gamma$ and either $(\sigma, \nu) \models
\varphi$ or $(\sigma, \nu) \models \psi$ or both.
In the former case $\Gamma \cup \{\varphi\} \proves \chi$ yields $(\sigma,
\nu) \models \chi$; and in the latter case $\Gamma \cup \{\psi\} \proves \chi$
yields $(\sigma, \nu) \models \chi$.
Therefore, $\vee$-E is sound.
The soundness proofs of the other five are immediate.
\end{proof}

The following lemma illustrates subproofs and proofs by cases, and gives
another five results that are needed later.
The last three of them are examples of the ability of the proof system of
exploiting the commutativity and associativity of $\vee$.

\begin{lemma}\label{L:D2-C6}
For every formula $\varphi$, $\psi$, and $\chi$, our proof system proves the
following:
\begin{description}

\item[D2] $\emptyset \proves \isdef{\varphi} \vee \neg \isdef{\varphi}$

\item[C6] $\emptyset \proves \TT$

\item[$\vee$-C]
$\{\varphi \vee \psi\} \proves \psi \vee \varphi$

\item[$\vee$-A]
$\{(\varphi \vee \psi) \vee \chi\} \proves \varphi \vee (\psi \vee \chi)$

\item[$\vee$-A']
$\{(\varphi \vee \psi) \vee \chi\} \proves \varphi \vee (\chi \vee \psi)$

\end{description}
\end{lemma}
\begin{proof}\mbox{}
\begin{enumerate}

\item[D2]
Since $\{\varphi\}$ $\prv{D1}$ $\isdef{\varphi}$ and $\{\neg \varphi\}$
$\prv{D1}$ $\isdef{\neg \varphi}$ $\samef_{\ref{D:isdef2}(\ref{D:id-neg})}$
$\isdef{\varphi}$, we get $\{\varphi \vee \neg \varphi\} \prv{$\vee$-E}
\isdef{\varphi}$ $\prv{$\vee$-I1}$ $\isdef{\varphi} \vee \neg
\isdef{\varphi}$.
On the other hand, $\{\neg \isdef{\varphi}\}$ $\prv{$\vee$-I2}$
$\isdef{\varphi} \vee \neg \isdef{\varphi}$.
Therefore, $\emptyset$~$\prv{C1}$ $(\varphi \vee \neg \varphi) \vee \neg
\isdef{\varphi}$ $\prv{$\vee$-E}$ $\isdef{\varphi} \vee \neg \isdef{\varphi}$.

\item[C6]
Since $\{\TT\}$ $\prv{P1}$ $\TT$ and $\{\neg \TT\}$ $\prv{D1}$ $\isdef{\neg
\TT}$ $\samef$ $\isdef{\TT}$ $\samef_{\ref{D:isdef2}(\ref{D:id-F-T})}$ $\TT$,
we reason $\emptyset$ $\prv{D2}$ $\isdef{\TT} \vee \neg \isdef{\TT}$ $\samef$
$\TT \vee \neg \TT$ $\prv{$\vee$-E}$ $\TT$.

\item[$\vee$-C]
$\{\varphi\}$ $\prv{$\vee$-I2}$ $\psi \vee \varphi$ and $\{\psi\}$
$\prv{$\vee$-I1}$ $\psi \vee \varphi$, thus $\{\varphi \vee \psi\}$
$\prv{$\vee$-E}$ $\psi \vee \varphi$.

\item[$\vee$-A]
$\{\varphi\}$ $\prv{$\vee$-I1}$ $\varphi \vee (\psi \vee \chi)$ and $\{\psi\}$
$\prv{$\vee$-I1}$ $\psi \vee \chi$ $\prv{$\vee$-I2}$ $\varphi \vee (\psi \vee
\chi)$, hence $\{\varphi \vee \psi\}$ $\prv{$\vee$-E}$ $\varphi \vee (\psi
\vee \chi)$.
We also have $\{\chi\}$ $\prv{$\vee$-I2}$ $\psi \vee \chi$ $\prv{$\vee$-I2}$
$\varphi \vee (\psi \vee \chi)$.
Therefore, $\{(\varphi \vee \psi) \vee \chi\}$ $\prv{$\vee$-E}$ $\varphi \vee
(\psi \vee \chi)$.

\item[$\vee$-A']
The proof is like the previous proof, but builds $\chi \vee \psi$ instead of
$\psi \vee \chi$.

\end{enumerate}
\end{proof}

\begin{lemma}\label{L:vee-permu}
If $\alpha, \beta, \gamma$ is any permutation of $\varphi, \psi, \chi$, then
$\{(\varphi \vee \psi) \vee \chi\}$ $\proves$ $\alpha \vee (\beta \vee
\gamma)$.
\end{lemma}
\begin{proof}
$\{(\varphi \vee \psi) \vee \chi\}$ $\prv{$\vee$-A}$ $\varphi \vee (\psi \vee
\chi)$ $\prv{$\vee$-C}$ $(\psi \vee \chi) \vee \varphi$ $\prv{$\vee$-A}$ $\psi
\vee (\chi \vee \varphi)$ $\prv{$\vee$-C}$ $(\chi \vee \varphi) \vee \psi$
$\prv{$\vee$-A'}$ $\chi \vee (\psi \vee \varphi)$ $\prv{$\vee$-C}$ $(\psi \vee
\varphi) \vee \chi$ $\prv{$\vee$-A}$ $\psi \vee (\varphi \vee \chi)$
$\prv{$\vee$-C}$ $(\varphi \vee \chi) \vee \psi$ $\prv{$\vee$-A}$ $\varphi
\vee (\chi \vee \psi)$ $\prv{$\vee$-C}$ $(\chi \vee \psi) \vee \varphi$
$\prv{$\vee$-A'}$ $\chi \vee (\varphi \vee \psi)$
\end{proof}

By Lemma~\ref{L:isdef}(\ref{L:id-sound2}) and
Definition~\ref{D:sem-formula}(\ref{D:sem-neg}), for every formula $\varphi$
and for each interpretation, precisely one of $\tulk{\varphi}$,
$\tulk{\neg\varphi}$ and $\tulk{\neg \isdef{\varphi}}$ yields $\TT$.
Our proof system reflects the same idea in the following way.
The set $\Gamma$ is \emph{inconsistent} if and only if $\Gamma \proves \FF$.
C3, C4, and C5 tell that if $\Gamma$ is consistent, then at most one of
$\varphi$, $\neg \varphi$, and $\neg \isdef{\varphi}$ can be proven.
The next lemma tells that if two of them lead to a contradiction, then the
third one can be proven.

\begin{lemma}\label{L:contradict}
Assume that $\alpha, \beta, \gamma$ is any permutation of $\varphi, \neg
\varphi, \neg \isdef{\varphi}$.
If $\Gamma \cup \{\alpha\} \proves \FF$, then $\Gamma \proves \beta \vee
\gamma$.
If $\Gamma \cup \{\alpha\} \proves \FF$ and $\Gamma \cup \{\beta\} \proves
\FF$, then $\Gamma \proves \gamma$.
\end{lemma}
\begin{proof}
$\emptyset$ $\prv{C1}$ $(\varphi \vee \neg \varphi) \vee \neg
\isdef{\varphi}$, from which Lemma~\ref{L:vee-permu} yields $\emptyset \proves
\alpha \vee (\beta \vee \gamma)$.
Assume $\Gamma \cup \{\alpha\}$ $\proves$ $\FF$.
Then $\Gamma \cup \{\alpha\}$ $\prv{C2}$ $\beta \vee \gamma$.
On the other hand, $\Gamma \cup \{\beta \vee \gamma\}$ $\prv{P1}$ $\beta \vee
\gamma$.
Thus $\Gamma \cup \{\alpha \vee (\beta \vee \gamma)\}$ $\prv{$\vee$-E}$ $\beta
\vee \gamma$, and $\Gamma$ $\prv{P3}$ $\beta \vee \gamma$.
If also $\Gamma \cup \{\beta\}$ $\proves$ $\FF$, then $\Gamma \cup \{\beta\}$
$\prv{C2}$ $\gamma$ and $\Gamma \cup \{\gamma\}$ $\prv{P1}$ $\gamma$.
So $\Gamma \cup \{\beta \vee \gamma\}$ $\prv{$\vee$-E}$ $\gamma$, and by
$\Gamma$ $\proves$ $\beta \vee \gamma$ we get $\Gamma$ $\prv{P3}$ $\gamma$.
\end{proof}

Next we discuss rule schemas for equality.
The first of them differs from classical binary first-order logic in that it
insists the term in question be defined.
Without this restriction, we could, for instance, prove $\frac{1}{0} =
\frac{1}{0}$ using $=$-1.
\begin{description}

\item[$=$-1] $\{\isdef{t}\} \proves (t = t)$\hfill(*)

\item[$=$-2] If $\acall{\varphi}{x}$ is a formula, $1 \leq i \leq
\arit(\varphi)$, and $t_i$ and $t'_i$ are free for $x_i$ in
$\acall{\varphi}{x}$, then $\{(t_i = t'_i), \acall{\varphi}{t}\} \proves
\varphi(t_1, \ldots, t'_i, \ldots, t_{\arit(\varphi)})$.

\end{description}
\begin{lemma}
$=$-1 and $=$-2 are sound.
\end{lemma}
\begin{proof}\mbox{}
\begin{itemize}

\item[$=$-1]
Assume $\sigma \models \idf$ and $(\sigma, \nu) \models \isdef{t}$.
The latter means that $\tulk{\isdef{t}}(\nu) \yield \TT$.
By Lemma~\ref{L:isdef}(\ref{L:id-sound1}), $\tulk{t}(\nu)$ is defined.
So $\tulk{(t = t)}(\nu) \yield_{\ref{D:sem-formula}(\ref{D:sem-=})} \TT$, that
is, $(\sigma, \nu)$ $\models$ $(t = t)$.

\item[$=$-2]
If $(\sigma, \nu) \models (t_i = t'_i)$, then by Definition
\ref{D:sem-formula}(\ref{D:sem-=}) both $\tulk{t_i}(\nu)$ and
$\tulk{t'_i}(\nu)$ are defined, and they yield the same value.
Because both $t_i$ and $t'_i$ are free for $x_i$ in $\varphi(x_i)$, the values
yielded by $\tulk{t_i}(\nu)$ and $\tulk{t'_i}(\nu)$ are treated in the same
way in $\tulk{\varphi(t_i)}(\nu)$ and $\tulk{\varphi(t'_i)}(\nu)$.
As a consequence, if $(\sigma, \nu) \models \varphi(t_i)$, then also $(\sigma,
\nu) \models \varphi(t'_i)$.

\end{itemize}
\end{proof}

The next lemma tells that our proof system allows the substitution of a term
for an equal term inside a defined term.
Furthermore, it can exploit symmetry and transitivity of $=$.

\begin{lemma}\label{L:equality}
For every function symbol $f$, every $1 \leq i \leq \arit(f)$, and every term
$t_i$, $t'_i$, $t$, $t'$, $t_1$, $t_2$, and $t_3$, our proof system proves the
following:
\begin{description}

\item[$=$-3] $\{(t_i = t'_i), \isdef{\acall{f}{t}}\} \proves (\acall{f}{t} =
f(t_1, \ldots, t'_i, \ldots, t_{\arit(f)}))$

\item[$=$-4] $\{(t = t')\} \proves (t' = t)$

\item[$=$-5] $\{(t_1 = t_2), (t_2 = t_3)\} \proves (t_1 = t_3)$

\end{description}
\end{lemma}
\begin{proof}\mbox{}
\begin{itemize}

\item[$=$-3]
Let $\varphi(x_i) \samef (\acall{f}{t} = f(t_1, \ldots, x_i, \ldots,
t_{\arit(f)}))$.
It contains no quantifiers, so $t_i$ and $t'_i$ are free for $x_i$ in it.
We have $\{\isdef{\acall{f}{t}}\}$ $\prv{$=$-1}$ $\varphi(t_i)$ and $\{(t_i =
t'_i), \varphi(t_i)\}$ $\prv{$=$-2}$ $\varphi(t'_i)$ $\samef$ $(\acall{f}{t} =
f(t_1, \ldots, t'_i, \ldots, t_{\arit(f)}))$.
The claim now follows by P2 and P3.

\item[$=$-4]
Let $\varphi(x) \samef (x = t)$.
We reason $\{(t = t')\}$ $\prv{D1}$ $\isdef{t = t'}$
$\samef_{\ref{D:isdef2}(\ref{D:id-=})}$ $\isdef{t} \wedge \isdef{t'}$
$\prv{$\wedge$-E1}$ $\isdef{t}$ $\prv{$=$-1}$ $(t = t)$ $\samef$ $\varphi(t)$.
Clearly $\{(t = t'), \varphi(t)\}$ $\prv{$=$-2}$ $\varphi(t')$ $\samef$ $(t' =
t)$.
The claim now follows by P3.

\item[$=$-5]
Let $\varphi(x) \samef (x = t_3)$.
By abuse of $\samef$ we reason $\{(t_1 = t_2), (t_2 = t_3)\}$ $\samef$ $\{(t_1
= t_2), \varphi(t_2)\}$ $\prv{$=$-4}$ $(t_2 = t_1)$ $\prv{$=$-2}$
$\varphi(t_1)$ $\samef$ $(t_1 = t_3)$.

\end{itemize}
\end{proof}

Rule schemas for quantifiers have one difference from classical binary
first-order logic:
$\forall$-E insists that the term $t$ must be defined.
This prevents us from deriving, for instance, $0 \cdot \frac{1}{0} = 0$ from
$\forall x\s (0 \cdot x = 0)$.
We will see that thanks to Corollary~\ref{C:reg}, $\exists$-I does not need a
similar condition.
\begin{description}

\item[$\forall$-E] If $t$ is free for $x$ in $\varphi(x)$, then
$\{\isdef{t}, (\forall x\s \varphi(x))\} \proves \varphi(t)$.\hfill(*)

\item[$\forall$-I] If $\Gamma \proves \varphi(x)$ and $x$ does not occur free
in $\Gamma$, then $\Gamma \proves (\forall x\s \varphi(x))$.

\item[$\exists$-I] If $t$ is free for $x$ in $\varphi(x)$, then
$\{\varphi(t)\} \proves (\exists x\s \varphi(x))$.

\item[$\exists$-E] If $\Gamma \cup \{\varphi(y)\} \proves \psi$ and $y$ does
not occur in $\Gamma$, $(\exists x\s \varphi(x))$, nor in $\psi$, then $\Gamma
\cup \{(\exists x\s \varphi(x))\} \proves \psi$.

\end{description}
\begin{lemma}
$\forall$-E, $\forall$-I, $\exists$-I, and $\exists$-E are sound.
\end{lemma}
\begin{proof}\mbox{}
\begin{itemize}

\item[$\forall$-E]
If $\sigma \models \idf$ and $(\sigma, \nu) \models \isdef{t}$, then
$\tulk{t}(\nu)$ is defined.
Let its value be denoted by $d$.
If also $(\sigma, \nu) \models (\forall x\s \varphi(x))$, then
\imdl{\varphi(x)}{[x := d]}.
Because $t$ is free for $x$, also \imdl{\varphi(t)}{} by Lemma~\ref{L:t-d}.
Thus $\forall$-E is sound.

\item[$\forall$-I]
If $(\sigma, \nu) \models \Gamma$ and $x$ does not occur free in $\Gamma$,
then for any $d \in \dom$ we have $(\sigma, \nu[x := d]) \models \Gamma$.
If also $\sigma \models \idf$ and $\Gamma \proves \varphi(x)$, then $(\sigma,
\nu[x := d]) \models \varphi(x)$.
This means that $(\sigma, \nu) \models (\forall x\s \varphi(x))$, implying
that $\forall$-I is sound.

\item[$\exists$-I]
Assume $(\sigma, \nu) \models \varphi(t)$.
Because $t$ is free for $x$ in $\varphi(x)$, if $\tulk{t}(\nu)$ is defined,
then \imdl{\varphi(x)}{[x := \tulk{t}(\nu)]} by Lemma~\ref{L:t-d}.
Therefore, $(\sigma, \nu) \models (\exists x\s \varphi(x))$.
Otherwise, $\tulk{t}(\nu)$ is undefined.
Because \imdl{\varphi(t)}{}, by Corollary~\ref{C:reg} \imdl{\varphi(x)}{[x :=
d]} for every $d \in \mathbb{D}$.
Because $\mathbb{D} \neq \emptyset$ by definition, this implies
\imdl{\varphi(x)}{[x := d]} for at least one $d \in \mathbb{D}$, that is,
$(\sigma, \nu) \models (\exists x\s \varphi(x))$.
So $\exists$-I is sound.

\item[$\exists$-E]
Assume $(\sigma, \nu) \models \Gamma \cup \{(\exists x\s \varphi(x))\}$.
So \imdl{\varphi(x)}{[x := d]} for at least one $d \in \mathbb{D}$.
Because $y$ does not occur in $(\exists x\s \varphi(x))$, $y$ is free for $x$
in $\varphi(x)$.
Thus also \imdl{\varphi(y)}{[y := d]}.
Because $y$ does not occur in $\Gamma$, $(\sigma, \nu) \models \Gamma$ implies
$(\sigma, \nu[y := d]) \models \Gamma$.
If also $\sigma \models \idf$ and $\Gamma \cup \{\varphi(y)\} \proves \psi$,
we get \imdl{\psi}{[y := d]}.
Because $y$ does not occur in $\psi$, we have \imdl{\psi}{}.
This completes the soundness proof of $\exists$-E.

\end{itemize}
\end{proof}

Altogether, there are only four differences from classical binary first-order
logic: two that make each closed formula yield precisely one of three truth
values instead of two; one that enforces that an undefined term is not equal
to anything; and one that reflects the principle that variables range over
defined values only.

\section{Existence of Models and Completeness}\label{S:Compl}

Our proofs for the model existence theorem and completeness theorem
mimic~\cite{Pap94}, Proposition 5.7, pp.\ 107--110, which we have adapted to
3-valued first-order logic and our proof system.
We have also attempted to clarify many technicalities.
Ultimately the proofs are based on the well-known construction by Leon
Henkin~\cite{Hen49}.

Let $\Gamma$ be a set of formulas.
Because it may be infinite, it is possible that every variable symbol occurs
in it.
However, the Henkin construction needs infinitely many additional variable
symbols.
Further headache is caused by the fact that a term may be not free for a
variable symbol in a formula.
In what follows we cannot rely on Lemma~\ref{L:var_change}, because the Henkin
construction uses the formula literally as it is.
Therefore, we will use a provably equivalent term instead that is
substitutable in the original formula.
So for each finite set of variable symbols, each term must have a provably
equivalent term that contains none of the variable symbols.
To deal with these, we introduce a new indexing for the variable symbols and
a modified $\Gamma$ called $\Gamma'$ as follows.

The function $\iota: \mathbb{Z}^+ \times \mathbb{Z}^+ \to \mathbb{Z}^+; (i,j)
\mapsto \frac{1}{2}(i+j-1)(i+j-2)+j$ is a bijection.
Therefore, we get an alternative indexing for the variable symbols by, for $i
\in \mathbb{Z}^+$ and $j \in \mathbb{Z}^+$, defining that $\var_{i,j}$ is the
same variable symbol as $\var_{\iota(i,j)}$.
The set $\Gamma'$ is obtained by replacing all variable symbols $\var_i$ in
$\Gamma$ by the variable symbols $\var_{2i,1}$, and then, for $j > 1$, adding
the formulas $(\var_{i,1} = \var_{i,j})$.
The function $\idf$ need not be modified, because of the following.
For each $f$, the only occurrence of $\isdef{f}$ in our formalism is in
Definition \ref{D:isdef2}(\ref{D:id-t}).
There the names of its variables are insignificant, because free variables
have been substituted by terms and bound variables have been chosen so that
the substitution is legal.

Let $\varphi(x)$ be any formula, $t$ any term, and $\var_{i,j}$ any variable
symbol that occurs in $t$ and becomes bound in $\varphi(t)$.
Because every formula is finite, only finitely many variable symbols occur in
$\varphi(x)$.
So there is some $k \in \mathbb{Z}^+$ such that $\var_{i,k}$ does not occur in
$\varphi(x)$.
Because $j = 1$ or $(\var_{i,1} = \var_{i,j}) \in \Gamma'$, and $k = 1$ or
$(\var_{i,1} = \var_{i,k}) \in \Gamma'$, the replacement of $\var_{i,j}$ by
$\var_{i,k}$ in $t$ results in a term that seems intuitively equivalent to
$t$.
Based on this idea, we will be able to work around the ``not free for $x$''
problem in the sequel.

The variable symbols $\var_{2i-1,1}$ are almost but not entirely unused,
because they occur in the formulas $(\var_{2i-1,1} = \var_{2i-1,j})$ in
$\Gamma'$.
To discuss this, we introduce new concepts.
A \emph{duplicate} of $\var_{i,1}$ is any $\var_{i,j}$ with $j > 1$.
Let $\Upsilon$ be any superset of $\Gamma'$.
By $\var_{i,1}$ is \emph{vacant} in $\Upsilon$ we mean that for $j \in
\mathbb{Z}^+$, every occurrence of $\var_{i,j}$ in $\Upsilon$ is in a formula
of the form $(\var_{i,1} = \var_{i,j})$.
All $v_{2i-1,1}$ are vacant in $\Gamma'$, but only those $v_{2i,1}$ are vacant
in $\Gamma'$ where $v_i$ does not occur in $\Gamma$.

\begin{lemma}\label{L:more_var}
Let $\Upsilon$ be a set of formulas such that $\Gamma' \subseteq \Upsilon$.
\begin{enumerate}

\item\label{L:mv-1}
Every variable symbol $x$ has infinitely many variable symbols $y$ such that
$\Gamma' \proves (x = y)$.

\item\label{L:mv-2}
If $\Upsilon \setminus \Gamma'$ is finite, then infinitely many variable
symbols are vacant in $\Upsilon$.

\item\label{L:mv-3}
If $i \in \mathbb{Z}^+$, $\Upsilon \proves \FF$, and for $j > 1$, no
$\var_{i,j}$ occurs in $\Upsilon$ except in $(\var_{i,1} = \var_{i,j})$, then
$\Upsilon$ has a finite subset $\Delta$ such that $\Delta \proves \FF$ and no
$\var_{i,j}$ with $j > 1$ occurs in $\Delta$.

\item\label{L:mv-4}
If $\Gamma' \proves \FF$, then $\Gamma \proves \FF$.

\item\label{L:mv-5}
If $(\idf, \Gamma')$ has a model (Definition \ref{D:id-models}(\ref{D:id-m})),
then $(\idf, \Gamma)$ has a model.

\end{enumerate}
\end{lemma}
\begin{proof}\mbox{}
\begin{enumerate}

\item
The claim follows from the $(\var_{i,1} = \var_{i,j})$ by $=$-4 and $=$-5.

\item
All the $v_{2i-1,1}$ are vacant in $\Gamma'$.
Each additional formula in $\Upsilon$ can make only a finite number of them
non-vacant, because formulas are finite.

\item
Because every proof is finite, the proof $\Upsilon$ $\proves$ $\FF$ only uses
a finite number of elements of $\Upsilon$.
There is thus a finite $\Upsilon' \subseteq \Upsilon$ such that $\Upsilon'$
$\proves$ $\FF$.
Because it is finite, it has a minimal subset $\Delta$ such that $\Delta$
$\proves$ $\FF$.
It remains to be proven that no $\var_{i,j}$ with $j > 1$ occurs in $\Delta$.
We prove it by deriving a contradiction from the assumption that $\Delta$
contains a formula of the form $(\var_{i,1} = \var_{i,j})$ where $j > 1$.

Let $\Delta' = \Delta \setminus \{(\var_{i,1} = \var_{i,j})\}$.
By definition, $\Delta' \cup \{(\var_{i,1} = \var_{i,j})\}$ $=$ $\Delta$
$\proves$ $\FF$.
Because $\emptyset \prv{C6; $\wedge$-I} \TT \wedge \TT$, we have $\Delta' \cup
\{\neg \isdef{\var_{i,1} = \var_{i,j}}\}$ $\prv{P1}$ $\neg \isdef{\var_{i,1} =
\var_{i,j}}$ $\samef_{\ref{D:isdef2}(\ref{D:id-=}, \ref{D:id-vc})}$ $\neg(\TT
\wedge \TT)$ $\prv{C3}$ $\FF$.
By Lemma~\ref{L:contradict} we have $\Delta'$ $\proves$ $\neg (\var_{i,1} =
\var_{i,j})$.
By construction, $\var_{i,j}$ does not occur in $\Delta'$.
Therefore, $\Delta'$ $\prv{$\forall$-I}$ $\forall x\s \neg (\var_{i,1} = x)$
$\prv{$\forall$-E}$ $\neg (\var_{i,1} = \var_{i,1})$ $\prv{$=$-1; C3}$ $\FF$,
where $\isdef{\var_{i,1}} \samef \TT$ by \ref{D:isdef2}(\ref{D:id-vc}).
This contradicts the minimality of $\Delta$.

\item
By (\ref{L:mv-3}), there is $\Delta \subseteq \Gamma'$ such that $\Delta
\proves \FF$ and no variable symbol of the form $\var_{i,j}$, where $i \geq 1$
and $j > 1$, occurs in $\Delta$.
It implies that $\Delta$ does not contain any $\var_{2i-1,1}$ either.
Because every proof is finite, $\Delta \proves \FF$ only uses a finite number
of variable symbols.
So all occurrences of variable symbols of other forms than $\var_{2i,1}$ in
$\Delta \proves \FF$ can be replaced by so far unused variable symbols of the
form $\var_{2i,1}$, resulting in a proof of $\FF$ from $\Delta$ that only uses
variable symbols of the form $\var_{2i,1}$.
Now replacing each $\var_{2i,1}$ by $\var_{i}$ results in a proof of $\FF$
from $\Gamma$.

\item
If $(\sigma, \nu') \models \Gamma'$, then $(\sigma, \nu) \models \Gamma$,
where for $i \in \mathbb{Z}^+$ we have $\nu(i) = \nu'(\iota(2i, 1))$.

\end{enumerate}
\end{proof}

It is our goal to prove that if $\Gamma \not\proves \FF$, then $(\idf,
\Gamma)$ has a model.
It follows from Lemma~\ref{L:more_var}(\ref{L:mv-4}) and (\ref{L:mv-5}) that
it suffices to prove that if $\Gamma' \not\proves \FF$, then $(\idf, \Gamma')$
has a model.
In that proof, we may assume what Lemma~\ref{L:more_var}(\ref{L:mv-1}),
(\ref{L:mv-2}), and (\ref{L:mv-3}) state.

Next we extend $\Gamma'$ so that the extended set $\Gamma_\omega$ is
consistent and for each formula $\varphi$, precisely one of $\varphi$, $\neg
\varphi$ and $\neg \isdef{\varphi}$ is in $\Gamma_\omega$.
Furthermore, for each formula of the form $\exists x\s \psi(x)$ in
$\Gamma_\omega$ it contains a formula of the form $\psi(y)$ as well, so that
later in this section we can appeal to $\psi(y)$ to justify $\exists x\s
\psi(x)$.
Similarly every $\neg \forall x\s \psi(x)$ in $\Gamma_\omega$ is accompanied
by $\neg \psi(y)$ for some $y$.
The $y$ are called \emph{Henkin witnesses}.
Technically they are free variables.
However, our eventual goal is to build a model $(\sigma, \nu)$ for $\Gamma'$,
and in it each free variable $\var_i$ will have a single value $\nu(i)$.
So the Henkin witnesses will eventually represent constant values.

The set $\Gamma_\omega$ is built by processing all formulas $\varphi$ (also
those that are in $\Gamma'$) in some order $\varphi_1$, $\varphi_2$, \ldots\
and forming a sequence $\Gamma_0$, $\Gamma_1$, $\Gamma_2$, \ldots\ of sets of
formulas as follows.
Let $\Gamma_0 \defeq \Gamma'$.
For $i > 0$, we construct $\Gamma_i$ from $\Gamma_{i-1}$ and $\varphi_i$
according to the first item in the list below whose condition is satisfied by
$\varphi_i$ (we will later show that at least one item matches).
Please notice that each $\Gamma_i \setminus \Gamma_{i-1}$ contains at most two
formulas, and thus each $\Gamma_i \setminus \Gamma'$ is finite.
(Conditions of Cases~\ref{C:case4} and~\ref{C:case5} mention facts that could
be derived from the failure of earlier conditions.
This is to simplify subsequent discussion.)

\begin{enumerate}

\item\label{C:case1}
If $\Gamma_{i-1} \cup \{\isdef{\varphi_i}\} \proves \FF$, then $\Gamma_i
\defeq \Gamma_{i-1} \cup \{\neg \isdef{\varphi_i}\}$.

\item\label{C:case2}
If $\Gamma_{i-1} \cup \{\varphi_i\} \not\proves \FF$ and $\varphi_i$ is of the
form $\exists x\s \psi(x)$, then $\Gamma_i \defeq \Gamma_{i-1}$ $\cup$
$\{\varphi_i, \psi(y)\}$, where $y$ is a variable symbol that is vacant in
$\Gamma_{i-1}$ and does not occur nor its duplicates occur in $\varphi_i$.
By Lemma~\ref{L:more_var}(\ref{L:mv-2}) such an $y$ exists.

\item\label{C:case3}
If $\Gamma_{i-1} \cup \{\neg \varphi_i\} \not\proves \FF$ and $\varphi_i$ is
of the form $\forall x\s \psi(x)$, then $\Gamma_i$ $\defeq$ $\Gamma_{i-1}$
$\cup$ $\{\neg \varphi_i,$ $\neg \psi(y)\}$, where $y$ is a variable symbol
that is vacant in $\Gamma_{i-1}$ and does not occur nor its duplicates occur in
$\varphi_i$ nor $\isdef{\psi(x)}$.
By Lemma~\ref{L:more_var}(\ref{L:mv-2}) such an $y$ exists.

\item\label{C:case4}
If $\Gamma_{i-1} \cup \{\varphi_i\} \not\proves \FF$ and $\varphi_i$ is not of
the form $\exists x\s \psi(x)$, then $\Gamma_i \defeq \Gamma_{i-1} \cup
\{\varphi_i\}$.

\item\label{C:case5}
If $\Gamma_{i-1} \cup \{\neg \varphi_i\} \not\proves \FF$ and $\varphi_i$ is
not of the form $\forall x\s \psi(x)$, then $\Gamma_i \defeq \Gamma_{i-1} \cup
\{\neg \varphi_i\}$.

\end{enumerate}

Clearly $\Gamma' = \Gamma_0 \subseteq \Gamma_1 \subseteq \Gamma_2 \subseteq
\cdots$.
We choose $\Gamma_\omega \defeq \Gamma_0 \cup \Gamma_1 \cup \ldots$.

\begin{lemma}\label{L:Gw_cons}
If $\Gamma' \not\proves \FF$, then $\Gamma_\omega \not\proves \FF$.
\end{lemma}
\begin{proof}
We prove first by induction that each $\Gamma_i$ is consistent.
The assumption $\Gamma' \not\proves \FF$ gives the base case $\Gamma_0
\not\proves \FF$.
The induction assumption is that $\Gamma_{i-1} \not\proves \FF$.
The induction step is divided into five cases according to how $\Gamma_i$ is
formed.
In Cases~\ref{C:case4} and~\ref{C:case5} $\Gamma_i \not\proves \FF$ by the
condition of the case.
In the remaining cases we assume $\Gamma_i \proves \FF$ and derive a
contradiction.

In Case~\ref{C:case1} $\Gamma_{i-1} \cup \{\isdef{\varphi_i}\} \proves \FF$.
If $\Gamma_i \proves \FF$, then $\Gamma_{i-1} \cup \{\neg \isdef{\varphi_i}\}
\proves \FF$.
These yield $\Gamma_{i-1}$ $\prv{D2}$ $\isdef{\varphi_i} \vee \neg
\isdef{\varphi_i}$ $\prv{$\vee$-E}$ $\FF$, contradicting the induction
assumption.

If in Case~\ref{C:case2} $\Gamma_i \proves \FF$, then $\Gamma_{i-1} \cup
\{(\exists x\s \psi(x)), \psi(y)\} \proves \FF$.
Because $y$ does not occur in $\exists x\s \psi(x)$, it is free for $x$ in
$\psi(x)$.
Therefore, $\{\psi(y)\}$ $\prv{$\exists$-I}$ $\exists x\s \psi(x)$.
Hence $\Gamma_{i-1} \cup \{\psi(y)\} \prv{P3} \FF$.
By \ref{L:more_var}(\ref{L:mv-3}) there is a finite subset $\Delta$ of
$\Gamma_{i-1}$ such that $\Delta \cup \{\psi(y)\} \proves \FF$ and $y$ does
not occur in $\Delta$.
Thus the assumptions of $\exists$-E hold such that $\Gamma$, $\varphi$, and
$\psi$ are $\Delta$, $\psi(y)$, and $\FF$, respectively.
We obtain $\Delta \cup \{\exists x\s \psi(x)\}$ $\prv{$\exists$-E}$ $\FF$.
So $\Gamma_{i-1} \cup \{\exists x\s \psi(x)\}$ $\prv{P2}$ $\FF$, contradicting
the condition of the case.

In Case~\ref{C:case3} we denote by $\Gamma'_i$ the set $\Gamma_{i-1} \cup
\{\neg \forall x\s \psi(x)\}$.
So $\Gamma_i = \Gamma'_i \cup \{\neg \psi(y)\}$.
If $\Gamma_i \proves \FF$, then $\Gamma'_i \cup \{\neg \psi(y)\} \proves \FF$.
Again, by \ref{L:more_var}(\ref{L:mv-3}) there is a finite subset $\Delta$ of
$\Gamma'_i$ such that $\Delta \cup \{\neg \psi(y)\} \proves \FF$ and $y$ does
not occur in $\Delta$.
Because $y$ does not occur in $\neg \forall x\s \psi(x)$, we may assume $(\neg
\forall x\s \psi(x)) \in \Delta$.
Lemma~\ref{L:contradict} yields
\begin{equation}\label{E:notneg}
\Delta \proves \psi(y) \vee \neg \isdef{\psi(y)}.
\end{equation}
On the other hand, $\Delta$ $\prv{D1}$ $\isdef{\neg \forall x\s \psi(x)}$
$\samef_{\ref{D:isdef2}(\ref{D:id-neg}, \ref{D:id-forall})}$ $(\forall x\s
\isdef{\psi(x)}) \vee \exists x\s (\isdef{\psi(x)} \wedge \neg \psi(x))$.
We show that both $\Delta \cup \{\forall x\s \isdef{\psi(x)}\}$ and $\Delta
\cup \{\exists x\s (\isdef{\psi(x)} \wedge \neg \psi(x))\}$ prove $\FF$.
They imply $\Delta$ $\prv{$\vee$-E}$ $\FF$ and $\Gamma'_i$ $\prv{P2}$ $\FF$,
which contradicts the condition of the case.

First consider $\Delta \cup \{\forall x\s \isdef{\psi(x)}\}$.
We have $\isdef{y} \samef \TT$ by \ref{D:isdef2}(\ref{D:id-vc}).
Because $y$ does not occur in $\isdef{\psi(x)}$, we get $\{\forall x\s
\isdef{\psi(x)}\}$ $\prv{$\forall$-E}$ $\isdef{\psi(y)}$.
Since $\{\isdef{\psi(y)}\} \cup \{\psi(y)\}$ $\prv{P1}$ $\psi(y)$ and
$\{\isdef{\psi(y)}\} \cup \{\neg \isdef{\psi(y)}\}$ $\prv{C3}$ $\FF$
$\prv{C2}$ $\psi(y)$, we obtain $\Delta \cup \{\forall x\s \isdef{\psi(x)}\}$
$\prv{(\ref{E:notneg}); $\vee$-E}$ $\psi(y)$ $\prv{$\forall$-I}$ $\forall y\s
\psi(y)$, since $y$ does not occur in $\Delta$ or $\isdef{\psi(x)}$.
Because $\psi(y)$ is obtained from $\psi(x)$ by substituting $y$ for $x$, we
have $\{\forall y\s \psi(y)\}$ $\prv{$\forall$-E}$ $\psi(x)$
$\prv{$\forall$-I}$ $\forall x\s \psi(x)$.
Therefore, $\Delta \cup \{\forall x\s \isdef{\psi(x)}\}$ $\proves$ $\forall
x\s \psi(x)$ $\prv{C3}$ $\FF$, because $(\neg \forall x\s \psi(x)) \in
\Delta$.

To deal with $\Delta \cup \{\exists x\s (\isdef{\psi(x)} \wedge \neg
\psi(x))\}$, let $z$ be a variable symbol that does not occur in it.
We obtain $\Delta$ $\prv{(\ref{E:notneg}); $\forall$-I}$ $\forall y\s (\psi(y)
\vee \neg \isdef{\psi(y)})$ $\prv{$\forall$-E}$ $\psi(z) \vee \neg
\isdef{\psi(z)}$.
Clearly $\{\isdef{\psi(z)} \wedge \neg \psi(z)\} \cup \{\psi(z)\}$ $\proves$
$\FF$ and $\{\isdef{\psi(z)} \wedge \neg \psi(z)\} \cup \{\neg
\isdef{\psi(z)}\}$ $\proves$ $\FF$, which implies $\Delta \cup
\{\isdef{\psi(z)} \wedge \neg \psi(z)\}$ $\prv{$\vee$-E}$ $\FF$.
Thus $\Delta \cup \{\exists x\s (\isdef{\psi(x)} \wedge \neg \psi(x))\}$
$\prv{$\exists$-E}$ $\FF$.

This completes the proof of $\Gamma_i \not\proves \FF$ for all $i \in
\mathbb{N}$.

We still have to prove that $\Gamma_\omega \not\proves \FF$.
If $\Gamma_\omega \proves \FF$, let $\Gamma'_\omega$ be the set of formulas in
$\Gamma_\omega$ that occur in the proof of the contradiction.
The set $\Gamma'_\omega$ is finite because all proofs are finite.
Therefore, there is such an index $i \in \mathbb{N}$ that $\Gamma'_\omega
\subseteq \Gamma_i$.
By construction, $\Gamma'_\omega \proves \FF$.
This and P2 imply $\Gamma_i \proves \FF$, which contradicts the earlier
result.
\end{proof}

\begin{lemma}\label{L:tasan1}
Assume $\Gamma' \not\proves \FF$.
For each formula $\varphi$, the set $\Gamma_\omega$ contains exactly one of
the formulas $\varphi$, $\neg \varphi$, and $\neg \isdef{\varphi}$.
\end{lemma}
\begin{proof}
If, when $\varphi_i$ is processed in the construction of $\Gamma_\omega$, none
of the conditions of Cases \ref{C:case2}, \ref{C:case3}, \ref{C:case4}, and
\ref{C:case5} holds, then by Lemma~\ref{L:contradict} $\Gamma_{i-1} \proves
\neg \isdef{\varphi_i}$, and thus the condition of Case~\ref{C:case1} holds.
So $\Gamma_\omega$ contains, for each $\varphi$, at least one of $\neg
\isdef{\varphi}$, $\varphi$, and $\neg \varphi$.
By Lemma~\ref{L:Gw_cons}, C3, C4, and C5, $\Gamma_\omega$ contains at most one
of $\varphi$, $\neg \varphi$, and $\neg \isdef{\varphi}$.
\end{proof}

\begin{lemma}\label{L:Gw_compl}
Assume $\Gamma' \not\proves \FF$.
Then $\Gamma_\omega \proves \varphi$ if and only if $\varphi \in
\Gamma_\omega$.
\end{lemma}
\begin{proof}
If $\varphi \in \Gamma_\omega$, then $\Gamma_\omega$ $\prv{P1}$ $\varphi$.

Assume $\Gamma_\omega$ $\proves$ $\varphi$.
If $(\neg \varphi) \in \Gamma_\omega$, then $\Gamma_\omega$ $\prv{C3}$ $\FF$,
which contradicts Lemma~\ref{L:Gw_cons}.
Therefore, $(\neg \varphi) \notin \Gamma_\omega$.
A similar reasoning using C4 yields $(\neg \isdef{\varphi}) \notin
\Gamma_\omega$.
By Lemma~\ref{L:tasan1} at least one of $\varphi$, $\neg \varphi$, $\neg
\isdef{\varphi}$ is in $\Gamma_\omega$.
So $\varphi \in \Gamma_\omega$.
\end{proof}

\begin{lemma}\label{L:tasan2}
Assume $\Gamma' \not\proves \FF$.
For each formula $\varphi$, the set $\Gamma_\omega$ contains exactly one of
$\isdef{\varphi}$ and $\neg \isdef{\varphi}$.
For each term $t$, the set $\Gamma_\omega$ contains exactly one of $\isdef{t}$
and $\neg \isdef{t}$.
\end{lemma}
\begin{proof}
If $\varphi \in \Gamma_\omega$ or $(\neg \varphi) \in \Gamma_\omega$, then
$\Gamma_\omega$ $\prv{D1}$ $\isdef{\varphi}$ $\samef$ $\isdef{\neg \varphi}$.
By Lemma~\ref{L:Gw_compl} we have $\isdef{\varphi} \in \Gamma_\omega$.
Otherwise, Lemma~\ref{L:tasan1} yields $(\neg \isdef{\varphi}) \in
\Gamma_\omega$.
By C3 we have $\{\isdef{\varphi}, (\neg \isdef{\varphi})\} \not\subseteq
\Gamma_\omega$, completing the proof of the first claim.

If $\isdef{t = t} \in \Gamma_\omega$, then $\isdef{t} \in \Gamma_\omega$,
because $\isdef{t = t}$ $\samef_{\ref{D:isdef2}(\ref{D:id-=})}$ $\isdef{t}
\wedge \isdef{t}$.
Otherwise, by the first claim $(\neg \isdef{t = t}) \in \Gamma_\omega$.
We use Lemma~\ref{L:contradict} to show $\{\neg \isdef{t = t}\}$ $\proves$
$\neg \isdef{t}$.
Clearly $\neg \isdef{t = t}$ $\samef$ $\neg(\isdef{t} \wedge \isdef{t})$, so
$\{\neg \isdef{t = t}\} \cup \{\isdef{t}\}$ $\prv{$\wedge$-I}$ $\isdef{t}
\wedge \isdef{t}$ $\prv{C3}$ $\FF$.
Furthermore, $\{\neg \isdef{t = t}\}$ $\prv{D1}$ $\isdef{\neg \isdef{t = t}}$
$\samef$ $\isdef{\isdef{t = t}}$ $\samef$ $\isdef{\isdef{t} \wedge \isdef{t}}$
$\samef_{\ref{D:isdef2}(\ref{D:id-and})}$ $(\isdef{\isdef{t}} \wedge
\isdef{\isdef{t}}) \vee (\isdef{\isdef{t}} \wedge \neg \isdef{t}) \vee
(\isdef{\isdef{t}} \wedge \neg \isdef{t})$, so $\{\neg \isdef{t = t}\} \cup
\{\neg \isdef{\isdef{t}}\}$ $\prv{$\vee$-E}$ $\FF$.
\end{proof}

By Lemma~\ref{L:tasan1}, we may unambiguously assign each formula $\chi$ a
truth value $\Tulk{\chi}$ as follows:
\begin{align}\label{E:Tulk}
\Tulk{\chi} \yield \TT\textnormal{, if } & \chi \in \Gamma_\omega\\
\Tulk{\chi} \yield \FF\textnormal{, if } & \neg\chi \in \Gamma_\omega
\nonumber\\
\Tulk{\chi} \yield \UU\textnormal{, if } & \neg\isdef{\chi} \in \Gamma_\omega
\nonumber
\end{align}
We show next that these truth values respect the propositional constants and
connectives.

\begin{lemma}\label{L:subf}
If $\Gamma' \not\proves \FF$, then $\Tulk{\FF} \yield \FF$, $\Tulk{\TT} \yield
\TT$, $\Tulk{\neg \varphi} \yield \neg \Tulk{\varphi}$, $\Tulk{\varphi \wedge
\psi} \yield \Tulk{\varphi} \wedge \Tulk{\psi}$, and $\Tulk{\varphi \vee \psi}
\yield \Tulk{\varphi} \vee \Tulk{\psi}$.
\end{lemma}
\begin{proof}
\mbox{}
\begin{itemize}

\item[$\TT$, $\FF$]
By C6 and Lemma~\ref{L:Gw_compl}, $\TT \in \Gamma_\omega$, so $\Tulk{\TT}
\yield \TT$.
Because $\neg \isdef{\FF}$ $\samef$ $\neg \TT$ by
Definition~\ref{D:isdef2}(\ref{D:id-F-T}), we cannot have $(\neg \isdef{\FF})
\in \Gamma_\omega$.
Lemmas~\ref{L:Gw_cons} and~\ref{L:Gw_compl} rule out $\FF \in \Gamma_\omega$.
The only remaining possibility is $(\neg\FF) \in \Gamma_\omega$, so
$\Tulk{\FF} \yield \FF$.

\item[$\neg \varphi$]
Let $\chi \samef \neg \varphi$.
If $\Tulk{\varphi} \yield \FF$, then $(\neg \varphi) \in \Gamma_\omega$.
That is, $\chi \in \Gamma_\omega$, so $\Tulk{\chi} \yield \TT$.
If $\Tulk{\varphi} \yield \UU$, then $(\neg \isdef{\varphi}) \in
\Gamma_\omega$.
By Definition~\ref{D:isdef2}(\ref{D:id-neg}) $\isdef{\chi} \samef
\isdef{\varphi}$.
So $(\neg \isdef{\chi}) \in \Gamma_\omega$ and $\Tulk{\chi} \yield \UU$.
Finally, let $\Tulk{\varphi} \yield \TT$, that is, $\varphi \in
\Gamma_\omega$.
If $\Tulk{\chi} \yield \TT$, then $\{\varphi, \neg \varphi\}$ $\subseteq$
$\Gamma_\omega$ $\prv{C3}$ $\FF$.
If $\Tulk{\chi} \yield \UU$, then $\{\varphi, \neg \isdef{\neg \varphi}\}$ $=$
$\{\varphi, \neg \isdef{\varphi}\}$ $\subseteq$ $\Gamma_\omega$ $\prv{C4}$
$\FF$.
Thus $\Tulk{\chi} \yield \FF$.

\item[$\varphi \wedge \psi$]
Let $\chi \samef \varphi \wedge \psi$.
If $\Tulk{\varphi} \yield \TT$ and $\Tulk{\psi} \yield \TT$, then $\{\varphi,
\psi\}$ $\subseteq$ $\Gamma_\omega$ $\prv{$\wedge$-I}$ $\varphi \wedge \psi$.
By Lemma~\ref{L:Gw_compl} $\Tulk{\varphi \wedge \psi} \yield \TT$.

By Definition~\ref{D:isdef2}(\ref{D:id-and}) $\isdef{\varphi \wedge \psi}$
$\samef$ $(\isdef{\varphi} \wedge \isdef{\psi}) \vee (\isdef{\varphi} \wedge
\neg \varphi) \vee (\isdef{\psi} \wedge \neg \psi)$.
If $\Tulk{\varphi} \yield \UU$ and $\Tulk{\psi} \yield \TT$, then $\{(\neg
\isdef{\varphi}), \psi\}$ $\subseteq$ $\Gamma_\omega$ and $\Gamma_\omega \cup
\{\isdef{\varphi \wedge \psi}\}$ $\prv{$\wedge$-E1; $\wedge$-E2; C3;
$\vee$-E}$ $\FF$.
So by Lemma~\ref{L:tasan2} $(\neg \isdef{\varphi \wedge \psi}) \in
\Gamma_\omega$, that is, $\Tulk{\varphi \wedge \psi} \yield \UU$.
Similar reasoning applies to $\Tulk{\varphi} \yield \TT$ and
$\Tulk{\psi} \yield \UU$, and to $\Tulk{\varphi} \yield \UU$ and $\Tulk{\psi}
\yield \UU$.

If $\Tulk{\varphi} \yield \FF$ then $\neg \varphi$ $\in$ $\Gamma_\omega$
$\prv{D1}$ $\isdef{\neg \varphi}$ $\samef$ $\isdef{\varphi}$
$\prv{$\wedge$-I}$ $\isdef{\varphi} \wedge \neg \varphi$ $\prv{$\vee$-I2;
$\vee$-I1}$ $\isdef{\varphi \wedge \psi}$, ruling out $\Tulk{\varphi \wedge
\psi} \yield \UU$.
Also $\Tulk{\varphi \wedge \psi} \yield \TT$ is impossible, because
$\{{\varphi \wedge \psi}\}$ $\prv{$\wedge$-E1}$ $\varphi$.
So $\Tulk{\varphi \wedge \psi} \yield \FF$.
Similarly, if $\Tulk{\psi} \yield \FF$, then $\Tulk{\varphi \wedge \psi}
\yield \FF$.

\item[$\varphi \vee \psi$]
Let $\chi \samef \varphi \vee \psi$.
If $\Tulk{\varphi} \yield \TT$, then $\varphi \in \Gamma_\omega$
$\prv{$\vee$-I1}$ $\varphi \vee \psi$, so $\Tulk{\varphi \vee \psi} \yield
\TT$.
Similarly if $\Tulk{\psi} \yield \TT$, then $\Tulk{\varphi \vee \psi} \yield
\TT$.

By Definition~\ref{D:isdef2}(\ref{D:id-or}) $\isdef{\varphi \vee \psi}$
$\samef$ $(\isdef{\varphi} \wedge \isdef{\psi}) \vee (\isdef{\varphi} \wedge
\varphi) \vee (\isdef{\psi} \wedge \psi)$.
Like above, it yields a contradiction with any combination of $\Tulk{\varphi}$
and $\Tulk{\psi}$, where one of them is $\UU$ and the other is $\UU$ or $\FF$.
Therefore, these combinations make $\Tulk{\varphi \vee \psi} \yield \UU$.

If $\Tulk{\varphi} \yield \FF$ and $\Tulk{\psi} \yield \FF$, then $(\neg
\varphi) \in \Gamma_\omega$ $\prv{D1}$ $\isdef{\varphi}$ and $(\neg \psi) \in
\Gamma_\omega$ $\prv{D1}$ $\isdef{\psi}$, so $\Gamma_\omega$
$\prv{$\wedge$-I}$ $\isdef{\varphi} \wedge \isdef{\psi}$ $\prv{$\vee$-I1;
$\vee$-I1}$ $\isdef{\varphi \vee \psi}$.
This rules out $\Tulk{\varphi \vee \psi} \yield \UU$.
Clearly $\Gamma_\omega \cup \{\varphi\}$ $\prv{C3}$ $\FF$ and $\Gamma_\omega
\cup \{\psi\}$ $\prv{C3}$ $\FF$, so $\Gamma_\omega \cup \{\varphi \vee \psi\}$
$\prv{$\vee$-E}$ $\FF$, ruling out $\Tulk{\varphi \vee \psi} \yield \TT$.
Therefore, $\Tulk{\varphi \vee \psi} \yield \FF$.

\end{itemize}
\end{proof}

Assuming that $\Gamma'$ is consistent, we next build a model for
$\Gamma_\omega$.
By Definition \ref{D:id-models}(\ref{D:id-m}), it has to obey Definition
\ref{D:nu} and \ref{D:id-models}(\ref{D:id-f}).
The former will be proven now, and the latter as Lemma~\ref{L:m-id}.

\begin{lemma}\label{L:model}
If $\Gamma' \not\proves \FF$, then there are $\sigma = (\dom, \itpr)$ and
$\nu$ such that $(\sigma, \nu) \models \Gamma_\omega$.
\end{lemma}
\begin{proof}
Let us define the set of \emph{defined terms} as $\mc{T}_\mathsf{def} = \{ t
\mid \isdef{t} \in \Gamma_\omega\}$.
If $t \in \mc{T}_\mathsf{def}$ and if $(t = t') \in \Gamma_\omega$ or $(t' =
t) \in \Gamma_\omega$, then also $t' \in \mc{T}_\mathsf{def}$, because $\{t_1
= t_2\}$ $\prv{D1}$ $\isdef{t_1 = t_2}$ $\samef$ $\isdef{t_1} \wedge
\isdef{t_2}$ by Definition~\ref{D:isdef2}(\ref{D:id-=}).
By Lemma~\ref{L:Gw_compl}, $=$-1, $=$-4, and $=$-5, the relation $\{ (t_1,
t_2) \mid (t_1 = t_2) \in \Gamma_\omega\}$ is an equivalence on
$\mc{T}_\mathsf{def}$.
For each $t \in \mc{T}_\mathsf{def}$, let $\Tulk{t}$ be the equivalence class
that $t$ belongs to.
That is, $\Tulk{t} = \{ t' \mid (t = t') \in \Gamma_\omega \}$.
We let $\mathbb{D}$ be the set of these equivalence classes, that is,
$\mathbb{D} = \{ \Tulk{t} \mid \isdef{t} \in \Gamma_\omega\}$.
It is non-empty, because each variable symbol is a defined term.
If $t \notin \mc{T}_\mathsf{def}$, we leave $\Tulk{t}$ undefined.
To summarize:
\begin{itemize}

\item
$\isdef{t} \in \Gamma_\omega$ if and only if $\Tulk{t} \in \dom$; furthermore,
every element of $\dom$ is of this form.

\item
$\isdef{t} \notin \Gamma_\omega$ if and only if $\Tulk{t}$ is undefined.

\item
$(t = t') \in \Gamma_\omega$ if and only if $\Tulk{t} = \Tulk{t'}$;
furthermore, then both $\Tulk{t}$ and $\Tulk{t'}$ are defined.

\end{itemize}

We need to define $\itpr$ and $\nu$ so that $(\sigma, \nu) \models
\Gamma_\omega$ (Definition \ref{D:nu}) and $\sigma \models \idf$ (Definition
\ref{D:id-models}(\ref{D:id-f})).
We will choose them so that we can prove by induction that
\begin{align}
\tulk{t}(\nu) = \Tulk{t} & \textnormal{ for each }t \in
\mc{T}_\mathsf{def}\textnormal{, and}\label{E:t-t-T1}\\
\tulk{t}(\nu)\textnormal{ is undefined} & \textnormal{ for each }t \notin
\mc{T}_\mathsf{def}\textnormal{.}\label{E:t-t-T2}
\end{align}

By Definition~\ref{D:isdef2}(\ref{D:id-vc}) and C6, variable and constant
symbols are defined.
For each $n \in \mathbb{Z}^+$ we let $\nu(n) = \Tulk{\var_n}$.
Then $\tulk{\var_n}(\nu) = \Tulk{\var_n}$ by \ref{D:sem-term}(\ref{D:sem-v}).
To make \ref{D:sem-term}(\ref{D:sem-c}) yield $\tulk{c}(\nu) = \Tulk{c}$, we
choose $\tulk{C} = \Tulk{c}$.
This is the base case of the induction proof.

Our next task is, for each function symbol $f$, to define $\tulk{f}$ so that
it makes $\tulk{t}(\nu)$ be undefined or $\Tulk{t}$ as appropriate.
Let $\Tulk{t_1} \in \dom$, \ldots, $\Tulk{t_{\arit(f)}} \in \dom$.
We choose
$$
\tulk{f}(\Tulk{t_1}, \ldots, \Tulk{t_{\arit(f)}}) \mathrel{}\left\{
\begin{array}{ll}
= \Tulk{\acall{f}{t}} & \textnormal{if } \isdef{\acall{f}{t}} \in
\Gamma_\omega\\
\textnormal{is undefined} & \textnormal{otherwise.}
\end{array}\right.
$$

To show that this definition does not depend on the choice of the $t_i \in
\Tulk{t_i}$, choose any $1 \leq i \leq \arit(f)$.
Let $t'_i \in \Tulk{t_i}$, that is, $(t_i = t'_i) \in \Gamma_\omega$.
We leave terms out for brevity; for instance, $f(t_i)$ means $\acall{f}{t}$.
If $\isdef{f(t_i)} \in \Gamma_\omega$, then $=$-3 yields $(f(t_i) = f(t'_i))
\in \Gamma_\omega$.
So $\Tulk{f(t_i)} = \Tulk{f(t'_i)}$.
If $\isdef{f(t_i)} \notin \Gamma_\omega$ then also $\isdef{f(t'_i)} \notin
\Gamma_\omega$, because otherwise $=$-3 yields $(f(t'_i) = f(t_i)) \in
\Gamma_\omega$, implying $\isdef{f(t_i)} \in \Gamma_\omega$.

We now complete the induction proof regarding $\tulk{t}(\nu)$ and $\Tulk{t}$.
There are three cases.
\begin{itemize}

\item
Assume that $\Tulk{t_i}$ is defined for every $1 \leq i \leq \arit(f)$, and
$\isdef{\acall{f}{t}} \in \Gamma_\omega$.
By Definition \ref{D:sem-term}(\ref{D:sem-f}), the induction assumption, and
the definition of $\tulk{f}$ we have $\tulk{\acall{f}{t}}(\nu)$ $=$
$\tulk{f}(\tulk{t_1}(\nu), \ldots, \tulk{t_{\arit(f)}}(\nu))$ $=$
$\tulk{f}(\Tulk{t_1}, \ldots, \Tulk{t_{\arit(f)}})$ $=$ $\Tulk{\acall{f}{t}}$.

\item
Assume that $\Tulk{t_i}$ is defined for every $1 \leq i \leq \arit(f)$, but
$\isdef{\acall{f}{t}} \notin \Gamma_\omega$.
By the definition of $\Tulk{t}$, $\Tulk{\acall{f}{t}}$ is undefined.
On the other hand, $\tulk{f}(\Tulk{t_1}, \ldots, \Tulk{t_{\arit(f)}})$ is
undefined by the definition of $\tulk{f}$.
By the induction assumption, $\Tulk{t_i} = \tulk{t_i}(\nu)$ for $1 \leq i \leq
\arit(f)$, so $\tulk{f}(\tulk{t_1}(\nu), \ldots, \tulk{t_{\arit(f)}}(\nu))$ is
undefined.
By \ref{D:sem-term}(\ref{D:sem-f}) $\tulk{\acall{f}{t}}(\nu)$ is undefined.

\item
Assume that $\Tulk{t_i}$ is undefined for some $1 \leq i \leq \arit(f)$.
We have $\isdef{t_i} \notin \Gamma_\omega$, from which $\isdef{\acall{f}{t}}
\notin \Gamma_\omega$ by \ref{D:isdef2}(\ref{D:id-t}), $\wedge$-E1, and
$\wedge$-E2.
By the definition of $\Tulk{t}$, $\Tulk{\acall{f}{t}}$ is undefined.
On the other hand, by the induction assumption $\tulk{t_i}(\nu)$ is undefined,
so $\tulk{\acall{f}{t}}(\nu)$ is undefined by \ref{D:sem-term}(\ref{D:sem-f}).

\end{itemize}
This completes the proof of (\ref{E:t-t-T1}) and (\ref{E:t-t-T2}).

In~(\ref{E:Tulk}) we defined that if $\varphi$ is a formula, then
$\Tulk{\varphi}$ yields $\TT$, $\FF$, or $\UU$ according to which one of
$\varphi$, $\neg \varphi$, and $\neg \isdef{\varphi}$ is in $\Gamma_\omega$.
We will show by induction that
\begin{center}
$\tulk{\varphi}(\nu) \yield \Tulk{\varphi}$ for every $\varphi$.
\end{center}
This will imply that if $\varphi \in \Gamma_\omega$, then
$\tulk{\varphi}(\nu)$ $\yield$ $\Tulk{\varphi}$ $\yield$ $\TT$.
As a consequence, $(\sigma, \nu) \models \Gamma_\omega$.

The base case of the induction consists of the atomic formulas.
Definition \ref{D:sem-formula}(\ref{D:sem-F-T}) and Lemma~\ref{L:subf} tell
that $\tulk{\FF}(\nu) \yield \FF \yield \Tulk{\FF}$ and $\tulk{\TT}(\nu)
\yield \TT \yield \Tulk{\TT}$.
Next we show $\tulk{(t_1 = t_2)}(\nu)$ $\yield$ $\Tulk{t_1 = t_2}$.
\begin{itemize}

\item
If $\Tulk{t_1}$ is undefined, then $\isdef{t_1} \notin \Gamma_\omega$.
We have $\isdef{t_1 = t_2} \samef (\isdef{t_1} \wedge \isdef{t_2}) \notin
\Gamma_\omega$ by Definition~\ref{D:isdef2}(\ref{D:id-=}) and $\wedge$-E1.
By Lemma~\ref{L:tasan2}, $(\neg\isdef{t_1 = t_2}) \in \Gamma_\omega$, that is,
$\Tulk{(t_1 = t_2)} \yield \UU$.
On the other hand, because $\Tulk{t_1}$ is undefined also $\tulk{t_1}(\nu)$ is
undefined by (\ref{E:t-t-T2}), so $\tulk{(t_1 = t_2)}(\nu) \yield \UU$ by
\ref{D:sem-formula}(\ref{D:sem-=}).
The same argument applies if $\Tulk{t_2}$ is undefined.

\item
If $\Tulk{t_1}$ and $\Tulk{t_2}$ are defined, then $\tulk{t_1}(\nu) =
\Tulk{t_1}$ and $\tulk{t_2}(\nu) = \Tulk{t_2}$ by (\ref{E:t-t-T1}).

If $\Tulk{t_1} = \Tulk{t_2}$, then $(t_1 = t_2) \in \Gamma_\omega$, that is,
$\Tulk{(t_1 = t_2)} \yield \TT$.
On the other hand, $\tulk{t_1}(\nu) = \tulk{t_2}(\nu)$, so by
\ref{D:sem-formula}(\ref{D:sem-=}) we have $\tulk{(t_1 = t_2)}(\nu)$ $\yield$
$\TT$.

If $\Tulk{t_1} \neq \Tulk{t_2}$, then $(t_1 = t_2) \notin \Gamma_\omega$.
Furthermore, $\isdef{t_1} \in \Gamma_\omega$ and $\isdef{t_2} \in
\Gamma_\omega$.
By $\wedge$-I we have $\isdef{t_1 = t_2} \in \Gamma_\omega$, implying
$(\neg\isdef{t_1 = t_2}) \notin \Gamma_\omega$.
By Lemma \ref{L:tasan1} we have $(\neg(t_1 = t_2)) \in \Gamma_\omega$, that
is, $\Tulk{(t_1 = t_2)} \yield \FF$.
On the other hand, $\tulk{t_1}(\nu) \neq \tulk{t_2}(\nu)$, so by
\ref{D:sem-formula}(\ref{D:sem-=}) we have $\tulk{(t_1 = t_2)}(\nu)$ $\yield$
$\FF$.

\end{itemize}

By Definition~\ref{D:isdef2}(\ref{D:id-R}), $\Tulk{\acall{R}{t}} \yield \UU$
(that is, $(\neg \isdef{\acall{R}{t}}) \in \Gamma_\omega$) precisely when at
least one of the $\Tulk{t_i}$ is undefined.
This is in harmony with \ref{D:sem-formula}(\ref{D:sem-R}), according to which
$\tulk{\acall{R}{t}}(\nu)$ $\yield$ $\UU$ if and only if at least one
$\tulk{t_i}(\nu)$ is undefined.
When all the $\Tulk{t_i}$ are defined, we define $\tulk{R} = \{ (\Tulk{t_1},
\ldots, \Tulk{t_{\arit(R)}}) \mid (\acall{R}{t}) \in \Gamma_\omega \}$.
This makes $\Tulk{\acall{R}{t}}$ yield $\TT$ and $\FF$ when it should by
\ref{D:sem-formula}(\ref{D:sem-R}).
By $=$-2, whether or not $(\acall{R}{t}) \in \Gamma_\omega$, is independent of
the choice of the representatives $t_1$, \ldots, $t_{\arit(R)}$ of the
equivalence classes.
Because $\acall{R}{x}$ is an atomic formula, it contains no bound variables,
so the ``free for $x_i$'' condition in $=$-2 is satisfied.
We have shown $\tulk{\acall{R}{t}}(\nu)$ $\yield$ $\Tulk{\acall{R}{t}}$.

The base case of the induction proof is now complete.
The induction assumption is that subformulas obey $\tulk{\varphi}(\nu)$
$\yield$ $\Tulk{\varphi}$.
Our proof of the induction step follows a pattern that we now illustrate with
$\wedge$.
We have $\tulk{\varphi \wedge \psi}(\nu) \yield \tulk{\varphi}(\nu) \wedge
\tulk{\psi}(\nu) \yield \Tulk{\varphi} \wedge \Tulk{\psi} \yield \Tulk{\varphi
\wedge \psi}$ by Definition~\ref{D:sem-formula}(\ref{D:sem-and}), the
induction assumption, and Lemma~\ref{L:subf}.
The cases $\neg$ and $\vee$ follow similarly from
\ref{D:sem-formula}(\ref{D:sem-neg}) and (\ref{D:sem-or}), respectively.

We still have to deal with the quantifiers.
By the induction assumption, (\ref{E:t-t-T1}), and Lemma~\ref{L:t-d} we have
the following:
\begin{align}\label{E:t-x}
\textnormal{If $\Tulk{t} \in \dom$ and $t$ is free for $x$ in $\psi(x)$, then
$\Tulk{\psi(t)} \yield \tulk{\psi(t)}(\nu) \yield \tulk{\psi(x)}(\nu[x :=
\Tulk{t}])$.}
\end{align}

If $\Tulk{\forall x\s \psi(x)} \yield \TT$, then $(\forall x\s \psi(x)) \in
\Gamma_\omega$.
Assume $\Tulk{t}$ is an arbitrary element of $\mathbb{D}$.
Then $\isdef{t} \in \Gamma_\omega$.
By Lemma~\ref{L:more_var}(\ref{L:mv-1}), $t \in \Tulk{t}$ can be chosen so
that it is free for $x$ in $\psi(x)$.
Then $\{\forall x\s \psi(x)\}$ $\prv{$\forall$-E}$ $\psi(t)$.
Therefore, $\psi(t) \in \Gamma_\omega$, that is, $\Tulk{\psi(t)} \yield \TT$.
So $\tulk{\psi(x)}(\nu[x := \Tulk{t}]) \yield \TT$ by~(\ref{E:t-x}).
By Definition~\ref{D:sem-formula}(\ref{D:sem-forall}) $\tulk{\forall x\s
\psi(x)}(\nu) \yield \TT$.

If $\Tulk{\forall x\s \psi(x)} \yield \FF$, then $(\neg \forall x\s \psi(x))
\in \Gamma_\omega$.
When $\forall x\s \psi(x)$ was dealt with in the construction of
$\Gamma_\omega$, Case~\ref{C:case1} was not chosen, because otherwise we would
have $\{(\neg \forall x\s \psi(x)), (\neg \isdef{\forall x\s \psi(x)})\}$
$\subseteq$ $\Gamma_\omega$ $\prv{C5}$ $\FF$.
Case~\ref{C:case2} was not chosen since $\forall x\s \psi(x)$ is of wrong form
for it.
The condition of Case~\ref{C:case3} was satisfied, as otherwise
$\Gamma_\omega$ would be inconsistent.
In Case~\ref{C:case3}, the formula $\neg \psi(y)$ was added to
$\Gamma_\omega$, making $\Tulk{\psi(y)} \yield \FF$.
By \ref{D:isdef2}(\ref{D:id-vc}) $y \in \mc{T}_\mathsf{def}$, so
by~(\ref{E:t-x}) $\tulk{\psi(x)}(\nu[x := \Tulk{y}]) \yield \FF$.
By Definition~\ref{D:sem-formula}(\ref{D:sem-forall}) $\tulk{\forall x\s
\psi(x)}(\nu) \yield \FF$.

If $\Tulk{\forall x\s \psi(x)} \yield \UU$, then $(\neg \isdef{\forall x\s
\psi(x)}) \in \Gamma_\omega$.
If $(\neg \psi(t)) \in \Gamma_\omega$ for any $t \in \mc{T}_\mathsf{def}$ that
is free for $x$ in $\psi(x)$, then $\Gamma_\omega$ $\prv{D1; $\wedge$-I}$
$\isdef{\neg \psi(t)} \wedge \neg \psi(t)$
$\prv{\ref{D:isdef2}(\ref{D:id-neg}); $\exists$-I}$ $\exists x\s
(\isdef{\psi(x)} \wedge \neg \psi(x))$ $\prv{$\vee$-I2}$ $\isdef{\forall x\s
\psi(x)}$ $\prv{C3}$ $\FF$, because by \ref{D:isdef2}(\ref{D:id-forall}),
$\isdef{\forall x\s \psi(x)}$ $\samef$ $(\forall x\s \isdef{\psi(x)}) \vee
\exists x\s (\isdef{\psi(x)} \wedge \neg \psi(x))$.
So there is no $\Tulk{t} \in \dom$ such that $\Tulk{\psi(t)} \yield \FF$, that
is, $\tulk{\psi(x)}(\nu[x := \Tulk{t}]) \yield \FF$.
By~\ref{D:sem-formula}(\ref{D:sem-forall}), to show $\tulk{\forall x\s
\psi(x)}(\nu) \yield \UU$, it remains to be proven that there is $t \in
\mc{T}_\mathsf{def}$ that is free for $x$ in $\psi(x)$ such that
$\tulk{\psi(x)}(\nu[x := \Tulk{t}]) \yield \UU$, that is, $\Tulk{\psi(t)}
\yield \UU$, that is, $(\neg \isdef{\psi(t)}) \in \Gamma_\omega$.
We have $(\forall x\s \isdef{\psi(x)}) \notin \Gamma_\omega$, because
otherwise $\Gamma_\omega$ $\prv{$\vee$-I1}$ $\isdef{\forall x\s \psi(x)}$
$\prv{C3}$ $\FF$.
On the other hand, $\emptyset$ $\prv{D2}$ $\isdef{\psi(x)} \vee \neg
\isdef{\psi(x)}$ $\prv{D1; $\vee$-E}$ $\isdef{\isdef{\psi(x)}}$
$\prv{$\forall$-I}$ $\forall x\s \isdef{\isdef{\psi(x)}}$ $\prv{$\vee$-I1}$
$\isdef{\forall x\s \isdef{\psi(x)}}$, so $(\neg \isdef{\forall x\s
\isdef{\psi(x)}}) \notin \Gamma_\omega$.
By Lemma~\ref{L:tasan1}, $(\neg \forall x\s \isdef{\psi(x)}) \in
\Gamma_\omega$.
It is of the form $(\neg \forall x\s \ldots) \in \Gamma_\omega$ that was
discussed above, so there is a variable symbol $y$ such that $(\neg
\isdef{\psi(y)}) \in \Gamma_\omega$, that is, $\tulk{\psi(x)}(\nu[x :=
\Tulk{y}]) \yield \Tulk{\psi(y)} \yield \UU$.

If $\Tulk{\exists x\s \psi(x)} \yield \TT$, then $(\exists x\s \psi(x)) \in
\Gamma_\omega$.
Consider the step $\varphi_i$ $\samef$ $\exists x\s \psi(x)$ in the
construction of $\Gamma_\omega$.
Case~\ref{C:case1} was not chosen, because otherwise both $\varphi_i \in
\Gamma_\omega$ and $(\neg \isdef{\varphi_i}) \in \Gamma_\omega$.
Because $(\exists x\s \psi(x)) \in \Gamma_\omega \not\proves \FF$,
Case~\ref{C:case2} was chosen.
There the formula $\psi(y)$ was added to $\Gamma_\omega$, making
$\Tulk{\psi(y)} \yield \TT$.
Thus $\tulk{\exists x\s \psi(x)}(\nu) \yield \TT$ by
Definition~\ref{D:sem-formula}(\ref{D:sem-exists}).

If $\Tulk{\exists x\s \psi(x)} \yield \FF$, then $(\neg \exists x\s \psi(x))
\in \Gamma_\omega$.
Assume $\Tulk{t}$ is an arbitrary element of $\mathbb{D}$.
By Definition~\ref{D:sem-formula}(\ref{D:sem-exists}) we get $\tulk{\exists
x\s \psi(x)}(\nu) \yield \FF$, if we show $\tulk{\psi(t)}(\nu) \yield \FF$.
We have $\isdef{t} \in \Gamma_\omega$.
By Lemma~\ref{L:more_var}(\ref{L:mv-1}), $t$ can be chosen so that it is free
for $x$ in $\psi(x)$ and $\isdef{\psi(x)}$.
We show $\Tulk{\psi(t)} \yield \FF$, that is, $(\neg \psi(t)) \in
\Gamma_\omega$, by ruling out the other two possibilities.
If $\psi(t) \in \Gamma_\omega$, then $\Gamma_\omega$ $\prv{$\exists$-I}$
$\exists x\s \psi(x)$ $\prv{C3}$ $\FF$.
The case $(\neg \isdef{\psi(t)}) \in \Gamma_\omega$ can be split to two via
$\Gamma_\omega$ $\prv{D1}$ $\isdef{\neg \exists x\s \psi(x)}$
$\samef_{\ref{D:isdef2}(\ref{D:id-exists})}$ $(\forall x\s \isdef{\psi(x)})
\vee \exists x\s (\isdef{\psi(x)} \wedge \psi(x))$.
We have $\Gamma_\omega \cup \{\forall x\s \isdef{\psi(x)}\}$
$\prv{$\forall$-E}$ $\isdef{\psi(t)}$ $\prv{C3}$ $\FF$.
Furthermore, $\Gamma_\omega \cup \{\exists x\s (\isdef{\psi(x)} \wedge
\psi(x))\}$ $\prv{$\exists$-E}$ $\FF$, since $\{(\isdef{\psi(z)} \wedge
\psi(z)), \neg \exists x\s \psi(x)\}$ $\prv{$\wedge$-E2; $\exists$-I}$ $\FF$.

If $\Tulk{\exists x\s \psi(x)} \yield \UU$, then $(\neg \isdef{\exists x\s
\psi(x)}) \in \Gamma_\omega$.
There cannot be any $\Tulk{t} \in \dom$ such that $\psi(t) \in \Gamma_\omega$
(where $t$ is chosen so that it is free for $x$ in $\psi(x)$), because
otherwise $\Gamma_\omega$ $\prv{D1; $\wedge$-I}$ $\isdef{\psi(t)} \wedge
\psi(t)$ $\prv{$\exists$-I}$ $\exists x\s (\isdef{\psi(x)} \wedge \psi(x))$
$\prv{$\vee$-I2}$ $\isdef{\exists x\s \psi(x)}$ $\prv{C3}$ $\FF$.
By Definition~\ref{D:sem-formula}(\ref{D:sem-exists}), it remains to be proven
that there is $\Tulk{t} \in \dom$ such that $(\neg \isdef{\psi(t)}) \in
\Gamma_\omega$.
We have $(\forall x\s \isdef{\psi(x)}) \notin \Gamma_\omega$, because
otherwise $\Gamma_\omega$ $\prv{$\vee$-I1}$ $\isdef{\exists x\s \psi(x)}$
$\prv{C3}$ $\FF$.
The rest of the proof is the same as in the case $\Tulk{\forall x\s \psi(x)}
\yield \UU$.
\end{proof}

\begin{lemma}\label{L:m-id}
The $\dom$ and $\itpr$ defined in the proof of Lemma~\ref{L:model} satisfy
$(\dom, \itpr) \models \idf$.
\end{lemma}
\begin{proof}
By Definition \ref{D:id-models}(\ref{D:id-f}) and the choice of $\dom$ in the
proof of \ref{L:model}, for every function symbol $f$ and $t_1 \in
\mc{T}_\mathsf{def}$, \ldots, $t_{\arit(f)} \in \mc{T}_\mathsf{def}$ we have
to show the following:
\begin{center}
$\tulk{\isdef{f}}(\Tulk{t_1}, \ldots, \Tulk{t_{\arit(\isdef{f})}}) \yield \TT$
if and only if $\tulk{f}(\Tulk{t_1}, \ldots, \Tulk{t_{\arit(f)}})$ is defined.
\end{center}

By the construction in the proof of \ref{L:model}, $\tulk{f}(\Tulk{t_1},
\ldots, \Tulk{t_{\arit(f)}})$ is defined if and only if $\isdef{\acall{f}{t}}
\in \Gamma_\omega$, that is, $\Tulk{\isdef{\acall{f}{t}}}$ $\yield$ $\TT$.
We have
\begin{center}
$\Tulk{\isdef{\acall{f}{t}}} \yield \Tulk{\icall{f}{t}} \yield
\tulk{\icall{f}{t}}(\nu) \yield$\\
$\tulk{\isdef{f}}(\tulk{t_1}(\nu), \ldots, \tulk{t_{\arit(\isdef{f})}}(\nu))
\yield \tulk{\isdef{f}}(\Tulk{t_1}, \ldots, \Tulk{t_{\arit(\isdef{f})}})$
\end{center}
by the following.
The first step follows from $\Tulk{t_1} \in \dom$, \ldots,
$\Tulk{t_{\arit(f)}} \in \dom$ and \ref{D:isdef2}(\ref{D:id-t}).
The second and last step hold because by the proof of \ref{L:model}, for every
formula $\varphi$ we have $\tulk{\varphi}(\nu) = \Tulk{\varphi}$ and for every
defined term $t$ we have $\tulk{t}(\nu) = \Tulk{t}$.
The third step follows from \ref{D:sem-formula}(\ref{D:sem-=}),
\ref{D:sem-formula}(\ref{D:sem-R}), and the fact that each $t_i$ is free for
$x_i$ in $\icall{f}{x}$ by \ref{D:isdef2}(\ref{D:id-t}).
\end{proof}

\begin{theorem}\label{T:main}
If $(\idf, \Gamma) \not\proves \FF$, then $(\idf, \Gamma)$ has a model.
If $(\idf, \Gamma) \models \varphi$, then $(\idf, \Gamma) \proves \varphi$.
\end{theorem}
\begin{proof}
If $\Gamma \not\proves \FF$, then Lemma~\ref{L:more_var}(\ref{L:mv-4}) yields
$\Gamma' \not\proves \FF$.
By Lemma~\ref{L:model} and \ref{L:m-id} $\Gamma_\omega$ has a model.
It is a model of $\Gamma'$ as well, because $\Gamma' \subseteq \Gamma_\omega$.
So $\Gamma$ has a model by Lemma~\ref{L:more_var}(\ref{L:mv-5}).

If $\Gamma \not\proves \varphi$, then by Lemma~\ref{L:contradict} there is
$\psi \in \{\neg \varphi, \neg \isdef{\varphi}\}$ such that $\Gamma \cup
\{\psi\}$ $\not\proves$ $\FF$.
By the previous claim $\Gamma \cup \{\psi\}$ has a model.
By Definition \ref{D:id-models}(\ref{D:id-|=}) it contradicts the assumption
$\Gamma \models \varphi$.
\end{proof}

\section{Discussion}\label{S:Discuss}

Let us first discuss the mimicking of other logics by ours.
Some logics~\cite{deN17,Leh01} use strict logical connectives; that is, if
$\varphi$ or $\psi$ or both are undefined, then also $\varphi \wedge \psi$ and
$\varphi \vee \psi$ are undefined.
They can be reduced to our system by interpreting them as shorthands for
$(\varphi \wedge \neg \varphi) \vee (\psi \wedge \neg \psi) \vee (\varphi
\wedge \psi)$ and $(\varphi \vee \neg \varphi)$ $\wedge$ $(\psi \vee \neg
\psi)$~$\wedge$ $(\varphi \vee \psi)$, respectively.
The $\forall x\s \varphi(x)$ and $\exists x\s \varphi(x)$ of
\cite{deN17,McC63} are our $(\forall x\s \varphi(x)) \vee \exists x\s
(\varphi(x) \wedge \neg \varphi(x))$ and $(\exists x\s \varphi(x)) \wedge
\forall x\s (\varphi(x) \vee \neg \varphi(x))$, respectively.

Some logics~\cite{McC63} (a variant in~\cite{deN17}) interpret $\varphi \wedge
\psi$ and $\varphi \vee \psi$ like $\varphi \wedge (\neg \varphi \vee \psi)$
and $\varphi \vee (\neg \varphi \wedge \psi)$, respectively, are interpreted
in our logic.
It corresponds to how ``and'' and ``or'' work in many programming languages,
assuming that $\UU$ represents program crash or undefined behavior.
In the case of ``and'', if $\varphi$ yields $\FF$, then $\FF$ is returned
without evaluating $\psi$; if the evaluation of $\varphi$ crashes, then the
evaluation of ``and'' has crashed; and if $\varphi$ yields $\TT$, then $\psi$
is evaluated resulting either in a crash or a truth value $\FF$ or $\TT$ that
is returned.
``Or'' is computed analogously.
In programming literature, this is often called ``short-circuit evaluation''.

Relation symbols $R'$ that are not defined everywhere can be simulated as
follows.
A new function symbol $f$ is introduced that is undefined precisely when
desired.
We choose $\param{R}{d} \in \tulk{R}$ when $\tcall{f}{d}$ is undefined, and
use $\acall{R}{x}$ $\wedge$ $(\acall{f}{x} = \acall{f}{x})$ in the place of
$R'$.

The sources~\cite{BB+05,GaL90,ScB99} introduce an if-then-else operator for
terms.
It is obviously useful for defining functions in a recursive fashion, which is
common practice in computer science.
The atomic formula $R(${if $\chi$ then $t_1$ else $t_2)$ can be treated as an
abbreviation of $(\chi \wedge R(t_1))$ $\vee$ $(\neg \chi \wedge R(t_2))$
$\vee$ $(\chi \wedge \neg \chi)$.
This was exemplified in~(\ref{E:split}), where $|x| = ($if $x < 0$ then $-x$
else $x)$, and $(\chi \wedge \neg \chi)$ was omitted because in this case it
always yields $\FF$.
This makes it also possible to introduce non-strict functions, such as $($if
$x = 0$ then $0$ else if $y = 0$ then $0$ else $x \cdot y)$, which is a
version of multiplication that yields $0$ also when one argument is $0$ and
the other argument is undefined.
(A function is strict if and only if an undefined argument always makes the
result undefined.)

Some logics contain a non-strict unary connective that inputs a truth value
and yields $\TT$ if the input is $\FF$ or $\TT$, and $\FF$ if the input is
$\UU$.
That is, it is otherwise like our $\isdef{\varphi}$, but it is a connective
while $\isdef{}$ is an abbreviation.
In~\cite{deN17} it is $\#\varphi$, in~\cite{GaL90,JoM94} it is $\Delta
\varphi$, and in~\cite{McC63} it is $*\varphi$.
Let us use $*$.
In the presence of $*$ and some way to express the constant $\UU$, any truth
function of arity $n+1$ (including the irregular ones) can be expressed
recursively as $(P_{n+1} \wedge *P_{n+1} \wedge \varphi_\TT) \vee (\neg
P_{n+1} \wedge *P_{n+1} \wedge \varphi_\FF) \vee (\neg *P_{n+1} \wedge
\varphi_\UU)$, where $\varphi_\TT$, $\varphi_\FF$, and $\varphi_\UU$ express
truth functions of arity $n$.
In our logic, this can be mimicked by using $\isdef{\varphi}$ instead of
$*\varphi$.
To obtain $\UU$ one may declare a unary function symbol $f$ with $\isdef{f}
\samef \FF$ (that is, $f$ is defined nowhere), and use $f(x) = f(x)$.
In this way, any propositional connective can be mimicked.

If the axiom system that is to be mimicked does not contain a natural
counterpart for $\isdef{f}$, then a new relation symbol $R_f$ may be
introduced and used as $\isdef{f}$.
This is not void of content, because information about $R_f$ may then follow
from other axioms.
Let us illustrate this with an example.
Assume $(\forall x\s (x = 0 \vee x \cdot \frac{1}{x} = 1)) \in \Gamma$.
We have $\{x \cdot \frac{1}{x} = 1\}$ $\prv{D1}$ $\isdef{x \cdot \frac{1}{x} =
1}$ $\prv{\ref{D:isdef2}(\ref{D:id-=}); $\wedge$-E1}$ $\isdef{x \cdot
\frac{1}{x}}$ $\prv{\ref{D:isdef2}(\ref{D:id-t}); $\wedge$-E1; $\wedge$-E2}$
$\isdef{\frac{1}{x}}$ $\samef$ $R_{\frac{1}{x}}(x)$.
Thus $\Gamma \cup \{\neg(x = 0)\}$ $\prv{$\forall$-E; $\vee$-E}$
$R_{\frac{1}{x}}(x)$.
Assume that originally $\neg *(\frac{1}{0} = \frac{1}{0})$ was used to express
that $\frac{1}{0}$ is undefined.
We mimic it by letting $(\neg \isdef{\frac{1}{0} = \frac{1}{0}}) \in \Gamma$.
Then $\Gamma$ $\proves$ $\neg R_{\frac{1}{x}}(0)$, because
$\{R_{\frac{1}{x}}(0)\}$ $\prv{}$ $\TT \wedge \TT \wedge R_{\frac{1}{x}}(0)$
$\samef$ $\isdef{\frac{1}{0}}$ $\prv{$=$-1}$ $\frac{1}{0} = \frac{1}{0}$
$\prv{D1}$ $\isdef{\frac{1}{0} = \frac{1}{0}}$.

The $E! t$, where $t$ is a term, of free logics can be mimicked with
$\isdef{t}$.
The $\exists x\s \varphi(x)$ of \cite{Leh01} can be mimicked with $\exists x\s
(\isdef{\varphi(x)} \wedge \varphi(x))$.

The observations above suggest that most, if not all, ternary first-order
logics for partial functions can be mimicked by our logic.
On the other hand, our completeness result can be generalized to also cover
$*$.
Because $\neg{*}(\frac{1}{0} = 0)$ is true but $\exists x\s \neg{*}(x = 0)$ is
false on real numbers, $\exists$-I is not sound in the presence of $*$.
Therefore, we replace it by ``If $t$ is free for $x$ in $\varphi(x)$, then
$\{\isdef{t}, \varphi(t)\} \proves \exists x\s \varphi(x)$''.
This makes the soundness proof go through despite the fact that due to $*$,
the logic is no longer regular.
Then $*$ is given proof rules as suggested by Definition~\ref{D:isdef2}.
It is not hard to check that every instance of $\exists$-I in our completeness
proof uses as $t$ either a variable symbol or a term such that $\Tulk{t} \in
\dom$, that is, $\isdef{t} \in \Gamma_\omega$.
Therefore, also the completeness proof goes through.

Let us now briefly discuss which familiar laws must be changed when switching
from binary to our logic.
We have already mentioned that C1, D1, $=$-1, $\forall$-E, and the modified
$\exists$-I differ from the binary case.
Of the 21 propositional laws in \cite[Table~6.3]{Hei95} that do not use
$\rightarrow$ or $\leftrightarrow$, only the following four fail in our logic:
$P \vee \neg P$ equals $\TT$ (Law of Excluded Middle), $P \wedge \neg P$
equals $\FF$ (Law of Non-contradiction), $P \wedge (\neg P \vee Q)$ equals $P
\wedge Q$, and $P \vee (\neg P \wedge Q)$ equals $P \vee Q$ (short-circuit
vs.~ordinary conjunction and disjunction).
All of the 14 quantifier laws in \cite[(7.1), (7.6), p.~380]{Hei95} that do
not use $\rightarrow$ or $\leftrightarrow$ are valid in our logic.

The situation with $\rightarrow$ is less clear, starting from the question
what it should mean.
All the 3-valued logics that we checked use either
Kleene's~\cite{BB+05,GaL90,JoM94}, strict~\cite{deN17,Leh01,McC63}, or
no~\cite{DMR08,Leh94} $\rightarrow$.
(Some denote it with $\Rightarrow$.)
In classical binary first-order logic, $\rightarrow$ is closely linked to
$\proves$ via Modus Ponens ($\{(\varphi \rightarrow \psi), \varphi\}$
$\proves$ $\psi$) and Deduction Theorem (if $\Gamma \cup \{\varphi\}$
$\proves$ $\psi$, then $\Gamma$ $\proves$ $\varphi \rightarrow \psi$).
Kleene's, {\L}ukasiewicz's, and strict implication satisfy the former, but, as
we now show, not the latter.
Let $c$ be a constant symbol, $f$ be a unary function symbol with $\isdef{f}
\samef \FF$, $\varphi \samef (f(c) = c)$, and $\psi \samef \neg(c = c)$.
Then in any model $\tulk{\varphi}(\nu) \yield \UU$ and $\tulk{\psi}(\nu)
\yield \FF$, and $\tulk{\varphi \rightarrow \psi}(\nu) \yield \UU$ by
Figure~\ref{F:truthtables} or strictness.
Furthermore, $\emptyset \cup \{\varphi\}$ $\prv{D1}$ $\isdef{\varphi}$
$\samef_{\ref{D:isdef2}(\ref{D:id-=}), (\ref{D:id-t})}$ $\isdef{c} \wedge
\isdef{f}(c) \wedge \isdef{c}$ $\samef_{\ref{D:isdef2}(\ref{D:id-vc})}$ $\TT
\wedge \FF \wedge \TT$ $\prv{$\wedge$-E1, $\wedge$-E2}$ $\FF$ $\prv{C2}$
$\psi$.
If Deduction Theorem holds, we get $\emptyset$ $\proves$ $\varphi \rightarrow
\psi$, demonstrating that the proof system is unsound.

Both Modus Ponens and Deduction Theorem hold in our logic if we define
$\varphi \rightarrow \psi$ as a shorthand for $\neg (\varphi \wedge
\isdef{\varphi}) \vee \psi$ or $\neg (\varphi \wedge \isdef{\varphi}) \vee
(\psi \wedge \isdef{\psi})$.
On the other hand, of the 8 laws in \cite[Table~6.3]{Hei95} that contain
$\rightarrow$, Kleene's implication violates only the law that $P \rightarrow
P$ equals $\TT$, while {\L}ukasiewicz's violates 3 laws, $\neg (\varphi \wedge
\isdef{\varphi}) \vee \psi$ violates 4, $\neg (\varphi \wedge \isdef{\varphi})
\vee (\psi \wedge \isdef{\psi})$ violates 5, and strict implication violates 6.
As a consequence, maintaining Deduction Theorem and maintaining as many
familiar practical laws as possible seem conflicting goals.
Therefore, we leave it open what $\rightarrow$ should mean in the context of
our logic.

The main message of this section is the following.
It is often possible to use $\isdef{}$ instead of $*$ without losing
completeness.
By doing so other connectives than $\neg$, $\wedge$, $\vee$, $\forall$, and
$\exists$ can be eliminated, facilitating some useful practical reasoning
methods.

\paragraph{Acknowledgements.}
We thank Esko Turunen for the help he gave in checking earlier versions of our
proofs; Cliff Jones, Scott Lehmann and Fred Schneider for helpful discussions
on the topic; and the anonymous reviewers for their hard work.
Our special thanks go to the reviewer who pointed out that for the proof
system to be recursive, also the function from function symbols to their
isdef-formulas must be recursive.

\bibliographystyle{plain}
\bibliography{logic3}

\end{document}